\documentclass{svjour3} %sn-jnl
\journalname{}
\smartqed  % flush right qed marks, e.g. at end of proof

\usepackage{lineno} % insert line number
\usepackage{amsfonts,amsmath,amssymb,mathtools,mathrsfs,bm,bbm}
\usepackage{graphicx}
\usepackage{bm,bbm}
\usepackage[title]{appendix}%
\usepackage{textcomp}%
\usepackage{manyfoot}%
\usepackage{booktabs}%
\usepackage{listings}%
\usepackage{float}
\usepackage{color}
\usepackage{algorithm}
\usepackage{algpseudocode}
\usepackage{hyperref}
\usepackage{savesym}
\usepackage{datetime}
\usepackage{nicematrix}
\usepackage{booktabs}
\usepackage[normalem]{ulem}
\usepackage{multirow}
\usepackage{framed} 
\usepackage{xcolor} 
\colorlet{shadecolor}{gray!20} 
\definecolor{mygreen}{rgb}{0,0.6,0}

\usepackage{tikz}           % Drawing package
\usepackage{tikz-cd}
\usepackage{pstricks,pst-node,pst-text,pst-3d,graphics, bm}
\usepackage{stmaryrd}
\usepackage{pgfplots}
\pgfplotsset{compat=1.17}

\def\e{\mathrm{e}}

\def\Mb{{\bm M}}

\def\sym{\mathrm{sym}}
\def\skew{\mathrm{skew}}
\def\tr{\mathrm{tr}}
\def\dist{\mathrm{dist}}
\def\res{\mathrm{res}}

\def\Rbnk{{\mathbb{R}^{2n\times 2k}}}
\def\Rbnn{{\mathbb{R}^{2n\times 2n}}}
\def\Spkn{{\mathrm{Sp}(2k,2n)}}

\def\Spn{{\mathrm{Sp}(2n)}}

\def\D{\mathrm{D}}
\def\d{\mathrm{d}}
\def\tan{\mathrm{tan}}

\def\diag{\mathrm{diag}}

\def\calX{\mathcal{X}}

\def\calR{\mathcal{R}}
\def\calP{{\cal P}}
\def\calT{{\cal T}}
\def\romT{\mathrm{T}}
\def\calM{{\cal M}}

\def\e{\mathrm{e}}
\def\hess{\mathrm{Hess}}
\def\grad{\mathrm{grad}}

\def\dist{\mathrm{dist}}
\def\R{\mathbb{R}}

\newcommand{\skewset}{{\cal S}_{\mathrm{skew}}}
\newcommand{\symset}{{\cal S}_{\mathrm{sym}}}
\newcommand{\proj}{\mathcal{P}_X^{}}
\newcommand{\PXp}{\Pi_X^{\perp}}
\newcommand{\projn}{\mathcal{P}_X^\perp}
\newcommand{\proje}{\mathcal{P}_{X,e}^{}}
\newcommand{\projM}{\mathcal{P}_{X,\Mb}^{}}

\newcommand{\projc}{\mathcal{P}_{X,c}^{}}

\newcommand{\rgrad}[1]{\mathrm{grad} f(#1)}
\newcommand{\rgrade}[1]{\mathrm{grad}_e f(#1)}
\newcommand{\rgradM}[1]{\mathrm{grad}_{\Mb} f(#1)}
\newcommand{\rgradc}[1]{\mathrm{grad}_c f(#1)}

\newcommand{\Hess}[2]{\mathrm{Hess} f(#1)[#2]}
\newcommand{\HessM}[2]{\mathrm{Hess}_{\Mb} f(#1)[#2]}
\newcommand{\Hesse}[2]{\mathrm{Hess}_e f(#1)[#2]}
\newcommand{\Hessc}[2]{\mathrm{Hess}_c f(#1)[#2]}
\newcommand{\Nablaf}[2]{\nabla^2 \bar{f}(#1)[#2]}

\hypersetup{colorlinks=true,linkcolor=blue,citecolor=blue}

\newcommand{\TX}{{\mathrm{T}_{X}}\Spkn}
\newcommand{\NX}{{\mathrm{T}_{X}^\perp}\Spkn}
\newcommand{\NXc}{{\mathrm{T}_{X}^{\perp,c}}\Spkn}

\newcommand{\NXM}{{\mathrm{T}_{X}^{\perp,\Mb}}\Spkn}
\newcommand{\mvec}[1]{\mathrm{vec}(#1)}
\newcommand{\mveck}[1]{\mathrm{veck}(#1)}
\definecolor{bingreen}{rgb}{0,0.6,0.2}
\definecolor{binblue}{rgb}{0.5,0,1}
\definecolor{orange}{rgb}{1,0.45,0}

\begin{document}
\title{Symplectic Stiefel manifold: tractable metrics, second-order geometry and Newton's methods\thanks{This work was first publicly presented at the GAMM Annual Meeting in Magdeburg, Germany, March 18-22, 2024. Part of this work was initiated when NTS was with Universit\"{a}t Augsburg and most of it was done when BG and NTS were visiting the Vietnam Institute for Advanced Study in Mathematics (VIASM) whose supports and hospitalities are gratefully acknowledged. BG was supported by the Young Elite Scientist Sponsorship Program by CAST and the National Natural Science Foundation of China (grant No.~12288201).}}

\titlerunning{Tractable metrics and Newton's methods on the symplectic Stiefel manifold}   
% if too long for running head

% \subtitle{Do you have a subtitle?\\ If so, write it here}

\author{Bin\,Gao \and Nguyen Thanh Son \and Tatjana\,Stykel}

\institute{Bin Gao \at
    State Key Laboratory of Scientific and Engineering Computing, Academy of Mathematics and Systems Science, Chinese Academy of Sciences, 100190 Beijing, China \\
    \email{gaobin@lsec.cc.ac.cn};
    \and Nguyen Thanh Son \at
    Department of Mathematics and Informatics, Thai Nguyen University of Sciences, 24118 Thai Nguyen, Viet Nam\\
    \email{ntson@tnus.edu.vn};
    \and Tatjana Stykel \at
    Institut f\"{u}r Mathematik and Centre for Advanced Analytics and Predictive Sciences (CAAPS), Universit\"{a}t Augsburg, Universit\"{a}tsstra\ss e 12a, 86159 Augsburg, Germany\\
    \email{stykel@math.uni-augsburg.de} 
}

\date{Received: date / Accepted: date}
% The correct dates will be entered by the editor

\maketitle

\begin{abstract}
    Optimization under the symplecticity constraint is an approach for solving various problems in quantum physics and scientific computing. Building on the results that this optimization problem can be transformed into an unconstrained problem on the symplectic Stiefel manifold, we construct geometric ingredients for Riemannian optimization with a~new family of Riemannian metrics called tractable metrics and develop Riemannian Newton schemes. The newly obtained ingredients do not only generalize several existing results but also provide us with freedom to choose a suitable metric for each problem. To the best of our knowledge, this is the first try to develop the explicit second-order geometry and Newton's methods on the symplectic Stiefel manifold. For the Riemannian Newton method, we first consider novel operator-valued formulas for computing the Riemannian Hessian of a~cost function, which further allows the manifold to be endowed with a weighted Euclidean metric that can provide a preconditioning effect. We then solve the resulting Newton equation, as the central step of Newton's methods, directly via transforming it into a~saddle point problem followed by vectorization, or iteratively via applying any matrix-free iterative method either to the operator Newton equation or its saddle point formulation. Finally, we propose a hybrid Riemannian Newton optimization algorithm that enjoys both global convergence and quadratic/superlinear local convergence at the final stage. Various numerical experiments are presented to validate the proposed methods.
\end{abstract}
 
%\PACS{32C25 \and 65K05} add later
	% \subclass{MSC code1 \and MSC code2 \and more}

\section{Introduction}
Appearing as a representation of a symplectic map	
between two finite-di\-men\-sio\-nal real symplectic vector spaces of dimensions $2k$ and $2n$, a~matrix $S \in \Rbnk$ with $1\leq k \leq n$ is called a~\emph{symplectic matrix} if there holds 
	\begin{equation}\label{eq:Sympl_Mat}
		S^TJ_{2n}S = J_{2k}\quad\mbox{ with }\quad J_{2n} = \begin{bmatrix}
			0&I_n\\-I_n&0
		\end{bmatrix},
	\end{equation}
where $I_n$ denotes the $n\times n$ identity matrix. It was proven in \cite[Prop.~3.1]{GSAS21} that the set of symplectic matrices, denoted by $\Spkn$, is a differentiable manifold called the \emph{symplectic Stiefel manifold}. In the case of square matrices, i.e., $k = n$, this set, denoted by $\Spn$, with the matrix multiplication additionally forms a~Lie group. It is worth noting that the symplectic Stiefel manifold is unbounded, e.g., $\left[\begin{smallmatrix}
			a&0 \\ 0&1/a
		\end{smallmatrix}\right]$ is a~symplectic matrix for any $a\in\mathbb{R}\setminus\{0\}$.  
	
An~optimization problem with the symplecticity constraint 
	\begin{equation}\label{eq:opt_prob}
		\min_{X\in \Spkn}f(X)
	\end{equation}
appears naturally in various applications. For instance, minimizing a trace function on the symplectic Stiefel manifold is the central step for computing symplectic eigenvalues and eigenvectors of symmetric positive-definite (spd) matrices \cite{SonAGS21,SonSt22}. It also helps to minimize the projection error in proper symplectic decomposition, a popular method for structure-preserving model reduction of Hamiltonian systems \cite{PengM16,BendZ22,GSS24}. In quantum physics, optimal control of symplectic gates can be formulated as a special case of \eqref{eq:opt_prob} with $k = n$, e.g.,  \cite{WuCR08,WuCR10}.
	
As the constraint nicely constitutes a differentiable manifold, one can reformulate the optimization problem as an unconstrained problem on the Riemannian manifold $\Spkn$ after equipping this manifold with an appropriate metric, e.g., canonical-like and Euclidean metrics~\cite{GSAS21,GSAS21a}. Accordingly, other geometric objects such as normal space, orthogonal projections, and the Riemannian gradient of the cost function can be constructed. Moreover, several retractions, which are indispensable in Riemannian optimization, have been derived in \cite{GSAS21,BendZ21,OviH23,GSS24,JenZ24}, 
and Riemannian gradient descent (RGD) algorithms for solving the minimization problem~\eqref{eq:opt_prob} have also been developed there. Recently, Riemannian conjugate gradient (RCG) methods have been adapted to the symplectic Stiefel manifold $\Spkn$ in~\cite{Sato23}, and a penalty method~\cite{XiaoLK24} has been applied to solving the minimization problem~\eqref{eq:opt_prob}. On the one hand, the examples presented in \cite{GSAS21,BendZ22,GSS24,JenZ24}  show the potential applications of these algorithms for solving different optimization problems with the symplecticity constraint. On the other hand, they also reveal that these first-order schemes often suffer from slow convergence at the final phase when the iterates are relatively close to the optimal solution. This observation is expected, as it has been shown in \cite[Thm.~4.5.6]{AbsiMS08}, that the RGD method in general converges at a~linear rate. The situation is even worse when the Euclidean Hessian of the cost function is far from good conditioning. 
	
To pursue fast convergence, we delve into two directions: constructing an~exqui\-site Riemannian metric to improve the performance of RGD methods at the early phase; and developing Newton's methods (in the neighborhood of local solutions) at the final phase to bring a high-order convergence rate over gradient methods. Combining these two directions can further lead to an efficient hybrid method. A~brief survey on the existing methods for Riemannian optimization on the symplectic Stiefel manifold is given in Table~\ref{tab:methods}. 

\vspace{-2mm}
\begin{table}[htbp]
    \centering
    \caption{Existing methods for Riemannian optimization on the symplectic Stiefel manifold\label{tab:methods}}
    \begin{tabular}{crrrr}
        \toprule
          & Geometry & Metric & Retraction & Method \\
        \midrule
        \cite{GSAS21,SonAGS21}
        &  1st order & canonical-like & quasi-geodesic, Cayley & RGD\\
        \cite{GSAS21a} &  1st order  & Euclidean & quasi-geodesic, Cayley & RGD\\
        \cite{GSS24} &  1st order  & Euclidean & SR decomposition & RGD\\
        \cite{BendZ21,BendZ22} &  1st order & pseudo-Riemannian & exponential, Cayley & RGD\\
        \cite{OviH23} &  1st order & canonical-like & a family including Cayley & RGD\\
        \cite{Sato23} & 1st order & canonical-like & Cayley & RCG\\\cmidrule{1-5}
        \cite{JenZ24} & 2nd order & right-invariant & Cayley & trust region\\ 
        this work & 2nd order & a family of tractable & SR decomposition, Cayley& Newton\\ 
        \bottomrule
    \end{tabular}
\end{table}
\vspace{-5mm}

\paragraph{Contribution}
In this paper, along the first direction, we introduce a~new family of metrics---so-called tractable metrics, which includes the known metrics, e.g., canonical-like and Euclidean metrics---and establish the necessary geometric ingredients for Riemannian optimization on the symplectic Stiefel manifold. This allows Riemannian preconditioning by freely choosing a~suitable metric depending on the cost function and in turn significantly accelerate the convergence of existing RGD schemes. 

Following the second direction, we investigate Riemannian Newton methods on the symplectic Stiefel manifold $\Spkn$. For general Riemannian manifolds, this topic has been discussed in several works, see, e.g., \cite{Smit94,AbsiMS08,Boumal23} and references therein. The main challenges of these methods are the computation of the Riemannian Hessian of the cost function and a~numerical procedure for solving the resulting Newton equation. These computations do not only depend on the
geometry of the manifold but also on the chosen metric. 
To the best of our knowledge\footnote{
When we were preparing the paper, we noticed that there was an independent work~\cite{JenZ24} available online, which considered the Riemannian Hessian on the symplectic Stiefel manifold under a~right-invariant metric. The differences between ours and~\cite{JenZ24} are several folds. Firstly, we propose a family of metrics, which includes the known metrics (canonical-like and Euclidean) as special cases. Secondly, we adopt a novel operator-valued formula for (explicitly) constructing the Riemannian Hessian in matrix form, which is able to match with more general metrics. Thirdly, we propose Newton-type methods with convergence guarantees. The work~\cite{JenZ24} developed, however, the Riemannian Hessian under a~right-invariant metric by involving the Christoffel symbols and by computing the derivative of the Riemannian gradient, which is computationally intricate and does not admit an explicit expression. In addition, it was concerned with Riemannian trust-region methods.}, 
the (explicit) second-order geometry of the symplectic Stiefel manifold has not been considered so far except for the case $k=n$, see \cite{BirtCC20}.
Although there was an~attempt in~\cite[Sect.~7.7]{Boumal23} to reach a~general formulation for Riemannian Hessians on general Riemannian manifolds, it was confined to the Euclidean metric only.
In order to address this issue for different Riemannian metrics, the~operator-valued framework recently presented in~\cite{Ngu23} turns out to be helpful. In that work, the main geometric tools for optimization are constructed based on operator-valued expressions for a~general embedded submanifold endowed with a~tractable metric, which can be represented via a~weighted Euclidean metric on the ambient space with a~varying spd weighting matrix. This fits well with our goal as it encompasses both the known canonical-like and (constant) weighted Euclidean metrics. Hence, we derive the explicit Riemannian Hessian with the aid of operator-valued expressions. When it is restricted to the Euclidean metric, the operator-valued formulas reduce to the known results for general Riemannian manifolds derived in, e.g., \cite[Sect.~5.3]{AbsiMS08} and \cite[Sect.~7.7]{Boumal23}. If it is further restricted to the special case $k=n$, we exactly recover the known Riemannian Hessian established in~\cite{BirtCC20} on the symplectic group $\Spn$. Moreover, one can consider a~weighted Euclidean metric, possessing a preconditioning effect by exploiting the second-order information of the cost function, e.g., \cite{ShuA23,GaoPY2023}.

To realize the benefit of the achieved Riemannian Hessian of the cost function on the symplectic Stiefel manifold, numerical solution of the Newton equation is, by no means, less important. For solving this equation, either direct or iterative methods can be used. We first convert the Newton equation to a~saddle point problem in operator form and then, via vectorization, derive an~explicit formula for its solution. However, such an~approach is computationally expensive in a~large-scale setting. Alternatively, the Newton equation or its saddle point formulation can approximately be solved by an~iterative method which does not require an~explicit construction of the coefficient matrix and relies instead on the computation of matrix-vector products or evaluation of the underline linear operator.

The Riemannian Newton method is well-known to deliver (at least) quadratic local convergence under some appropriate conditions, see \cite{Smit94}, \cite[Thm.~6.3.2]{AbsiMS08} and \cite[Thm.~6.7]{Boumal23} for detail. Nevertheless, unlike the RGD 
method, the global convergence is not guaranteed. To deal with this, i.e., to combine the global convergence with a~fast convergence rate in one algorithm,  the Riemannian trust-region method has been developed in, e.g., \cite{AbsiBG07,JenZ24} which locally relaxes the Newton equation to the optimization of a~second-order approximate model. Other Riemannian Newton or Newton-type methods, that ensure global convergence, with specific applications can be found in \cite{ZhaoBJ15,ZhaoBJ18,BortFFY20,BortFF22,XuNgB22}, to name a~few. A common feature of these algorithms is that a~criterion was used to carefully switch between a~gradient descent method and a~Newton iteration combined with a~damping to ensure a~sufficient decrease in the cost function or a merit function. 
Inspired by \cite{SatoI13,IzmaS14}, we follow a~slightly different approach and develop a~hybrid Riemannian Newton method on the symplectic Stiefel manifold, where the early stage employs a~RGD method from~\cite{GSAS21} and the final stage carries out the proposed Riemannian Newton methods. Moreover, we prove its global convergence and local convergence rates based on Newton and inexact Newton iteration. Note that the analysis is able to encompass general manifolds and can be naturally extended to optimization on other manifolds as well as to the detection of a~singularity point of a~vector field. Except for \cite[Prop.~8]{RingW12}, which requires the equicontinuity of the derivative of the retraction used, to the best of our knowledge, this is the first attempt to derive the convergence properties of inexact Newton method on its own with standard assumption and as a part of a hybrid algorithm, which enables global convergence, on Riemannian manifolds. 

In order to validate the performance of the proposed Riemannian optimization methods, we consider the problem of finding symplectic solutions of a~matrix least squares problem and minimization of trace cost functions. The numerical experiments show that both preconditioned RGD schemes and Riemannian Newton algorithms outperform the existing schemes and that the proposed hybrid Riemannian Newton method converges faster to a solution at the final stage regardless of the starting point. 

\paragraph{Organization} 
After introducing the notation, the rest of this paper is organized as follows. 
In section~\ref{sec:Geometry_sympl_Stiefel}, we consider the geometry of the symplectic Stiefel manifold by introducing a~new family of Riemannian metrics for which both the canonical-like and Euclidean metrics can be considered as a special case. Notably, important geometric tools---such as orthogonal projections onto the tangent and normal spaces, and the Riemannian gradient---can be formulated using operator-valued framework. In section~\ref{sec:Rhessian}, we derive the Riemannian Hessian formulas corresponding to both the canonical-like and weighted Euclidean metrics. In section~\ref{sec:Solve_NewtonEq}, we develop Riemannian Newton methods on the symplectic Stiefel manifold and discuss how to solve the Newton equation in detail. The inexact and hybrid Riemannian Newton methods are also presented. The convergence properties of the proposed algorithms are studied in section~\ref{sec:Convergence}. Numerical examples are provided in section~\ref{sec:Numer} and concluding remarks are given in section~\ref{sec:Conclusion}.   

\paragraph{Notation} 
We use $\symset(m)$ and $\skewset(m)$ to denote the sets of all $m\times m$ real symmetric and skew-symmetric  matrices, respectively. For a square matrix $A$, $\sym(A)=\frac{1}{2}(A+A^T)$ and  $\skew(A)=\frac{1}{2}(A-A^T)$ denote its symmetric and skew-symmetric part, respectively, and $\tr(A)$ denotes the trace of $A$. We use $\frac{\d}{\d t}(\cdot)$, $\D (\cdot)$ and $\D_Z(\cdot)$ to denote the classical derivative of a function of one real variable $t$, the Fr\'{e}chet derivative of a~mapping between two Euclidean spaces, and the directional derivative along a~vector~$Z$, respectively. For a smooth function $f$ defined on the embedded submanifold of interest, $\nabla \bar{f}(X)$ and $\nabla^2 \bar{f}(X)$ represent, respectively, the classical gradient and Hessian of a~smooth extension $\bar{f}$ of $f$ to the ambient space of the submanifold in a~neighborhood of~$X$, or the ambient gradient and ambient Hessian for short. The standard Euclidean metric defined on the ambient space is denoted by~$\langle\cdot,\cdot\rangle$; this notation can be accompanied by subscripts depending on the context. Furthermore, we use $\|\cdot\|_{\mathrm{F}}$ to denote the Frobenius matrix norm. 

%---------------------------------------------------------------------

\section{Riemannian geometry under tractable metrics}
\label{sec:Geometry_sympl_Stiefel}
In this section, we briefly review the symplectic Stiefel manifold
\[
\Spkn=\big\{X\in\Rbnk\ : \ X^TJ_{2n}X = J_{2k}\big\},
\] 
by not only collecting necessary results from \cite{GSAS21a,GSAS21} but also introducing a~family of Riemannian metrics on $\Spkn$ which generalizes both the canonical-like and Euclidean metrics studied in \cite{GSAS21a,GSAS21}. The corresponding geometric ingredients can then be derived by using an~operator-valued framework developed in \cite{Ngu23}, which is crucial for constructing the Riemannian Newton method.

For the purpose of introducing operator-valued operations, let us first recall that $\Spkn$ is a~closed embedded submanifold of the Euclidean space $\R^{2n\times 2k}$ of dimension \mbox{$4nk-k(2k-1)$}. This fact immediately follows from the submersion theorem \cite[Prop.~3.3.3]{AbsiMS08} applied to a~smooth mapping	
\[
\arraycolsep=2pt
\begin{array}{rccc}
F\;: \;\;&\Rbnk&\longrightarrow &\skewset(2k)\\ 
&X&\longmapsto&  X^TJ_{2n}X - J_{2k},
\end{array}
\]
which determines the symplectic Stiefel manifold $\Spkn$, see \cite[Prop.~3.1]{GSAS21}. The Fr\'{e}chet derivative of $F$ at $X\in \Spkn$  is
given by
\begin{equation}\label{eq:DF}
\arraycolsep=2pt
\begin{array}{rccc}
	\D F_X\; :\;\;&\Rbnk&\longrightarrow &\skewset(2k)\\ 
	&Z&\longmapsto& X^TJ_{2n}Z + Z^TJ_{2n}X.
\end{array}
\end{equation}
The kernel of $\D F_X$ defines the \emph{tangent space} $\TX$ to $\Spkn$ at $X$, i.e.,
\begin{equation}\label{eq:tangentspace_1}
\TX = \ker(\D F_X) = \big\{Z \in \Rbnk\ :\ X^TJ_{2n}Z + Z^TJ_{2n}X = 0\big\}.
\end{equation}
This space can also be characterized as 
\begin{equation}
\label{eq:tangentspace_2}
\TX =\big\{XJ_{2k}W + J_{2n}X_{\perp} K\ :\  W\in \symset(2k),\ K\in \R^{(2n-2k)\times 2k}\big\},
\end{equation}
where the orthogonal complement $X_\perp\in\R^{2n\times(2n-2k)}$ has full rank and satisfies the relation $X^TX_\perp=0$.

The adjoint operator to  $\D F_X$ with respect to the Euclidean metric is defined as the linear operator $\D F_X^* : \skewset(2k)\rightarrow \Rbnk$ which for all $Z\in\Rbnk$ and $\varOmega\in\skewset(2k)$, satisfies 
\[
\langle \D F_X(Z),\varOmega\rangle_{\skewset(2k)}= \langle Z,\D F_X^*(\varOmega)\rangle_{\Rbnk}.
\]
A~direct calculation using $\varOmega^T=-\varOmega$ and $J_{2n}^T=-J_{2n}$ yields
\[
\langle \D F_X(Z),\varOmega\rangle_{\skewset(2k)} = 2\,\tr(Z^TJ_{2n}^TX\varOmega)=\langle Z, 2J_{2n}^TX\varOmega\rangle_{\Rbnk},
\] 
which implies that
\begin{equation}\label{eq:DF_star_detail}
\D F_X^*(\varOmega) = 2\,J_{2n}^TX\varOmega = -2\,J_{2n}^{}X\varOmega.
\end{equation}
It is ready to know that the adjoint operator $\D F_X^*$ is injective. 

%-----------------------------------------------------------------------------%
%
\subsection{Tractable metrics}
Equipping $\TX$ with an inner product $\langle\cdot,\cdot\rangle_X$, which varies
smoothly with~$X$, we turn $\Spkn$ into a Riemannian manifold. Inspired by \cite{Ngu23,ShuA23},  for an~spd matrix $\Mb_X\in\R^{2n\times 2n}$ smoothly depending on~$X$, we introduce a family of Riemannian metrics on $\Spkn$, so-called \emph{tractable metrics}, induced by the standard Euclidean metric  $\langle\cdot,\cdot\rangle$ on the ambient space $\R^{2n\times 2k}$ as 
\begin{equation}\label{eq:metricM}
g_{\Mb_X}(Z_1,Z_2):= \langle Z_1,\Mb_{X}Z_2\rangle = \tr(Z_1^T\Mb_{X}Z^{}_2),
\qquad Z_1,Z_2\in\TX.
\end{equation} 
It induces the norm $\|Z\|_{\Mb_X}=\sqrt{g_{\Mb_X}\!(Z,Z)}$ for a~tangent vector $Z\in\TX$. For a~linear operator $\mathcal{A}:\TX\to \TX$, we define the operator norm induced by the metric $g_{\Mb_X}$ as 
\[
\|\mathcal{A}\|_{\Mb_X} 
    := \sup\big\{\|\mathcal{A}(Z)\|_{\Mb_X}\ :\ Z\in \TX,\ \|Z\|_{\Mb_X} = 1\big\}. 
\]

Furthermore, the \emph{normal space} to $\Spkn$ at $X$ with respect to a~tractable metric $g_{\Mb_X}$ is defined by 
\[
\NX := \bigl\{ N\in\mathbb{R}^{2n\times 2k}\; : \; g_{\Mb_X}(N,Z) 
= 0 \text{ for all } Z\in \TX\bigr\}.
\]
This space can be characterized as follows.

\begin{proposition}\label{prop:normalMX}
The normal space to $\Spkn$ at $X\in\Spkn$ with respect to the metric $g_{\Mb_X}\!$ defined in~\eqref{eq:metricM}
can be represented as
\begin{equation}\label{eq:normal}
	\NX= \big\{\Mb_{X}^{-1}J_{2n}X\varOmega \; : \; \varOmega\in\skewset(2k)\big\}.
\end{equation}
\end{proposition}
\begin{proof}		
Due to \cite[Lem.~3.2]{GSAS21}, the matrices $E=[XJ_{2k},\; J_{2n}X_{\perp}]$ and 
\[
E^TJ_{2n}E=\begin{bmatrix}J_{2k} & 0 \\ 0& X_{\perp}^TJ_{2n}X_{\perp}^{}\end{bmatrix}
\]
are both nonsingular in $\Rbnn$. Then any $N\in\R^{2n\times 2k}$ can be represented as 
\[
N=\Mb_{X}^{-1}J_{2n}E(E^TJ_{2n}E)^{-1}\begin{bmatrix}\varOmega \\ K_N\end{bmatrix}
\] 
with $\varOmega\in\R^{2k\times 2k}$ and $K_N\in\R^{(2n-2k)\times 2k}$. Furthermore, taking into account \eqref{eq:tangentspace_2},  any $Z\in\TX$ can be written as 
$Z=E\left[\begin{smallmatrix}
    W\\K_Z
\end{smallmatrix}\right]$
with $W\in \symset(2k)$ and $K_Z\in\R^{(2n-2k)\times 2k}$.
Then the normal space condition $g_{\Mb_X} (N,Z)=0$ implies that
\[
0=\tr\big(N^T\Mb_{X}Z\big) 
=  \tr\big(\varOmega^TW\big) +\tr\big(K_N^TK_Z^{}\big)
\] 
for all $W\in \symset(2k)$ and $K_Z\in\R^{(2n-2k)\times 2k}$. 
This is equivalent to $\varOmega\in\skewset(2k)$ and $K_N=0$. Thus, \eqref{eq:normal} holds. \qed  
\end{proof}

Any matrix $Y\in\Rbnk$ can additively be decomposed as $$Y=\proj(Y)+\projn(Y),$$  
where $\proj$ and $\projn$ denote the \emph{orthogonal projections} onto the tangent and normal spaces, respectively.  Using  $\D F_X$ and its adjoint  $\D F_X^*$,  the projection  $\proj$ with respect to a tractable metric $g_{\Mb_X}$ can be represented as
\begin{equation}\label{eq:Du_orth_proj}
\proj(Y) = Y-\Mb_X^{-1}\D F_X^*(\D F_{X}\Mb_X^{-1}\D F_X^*)^{-1}\D F_{X}(Y),
\end{equation}
see~\cite[Prop.~3.1]{Ngu23}. Note that the invertibility of $\D F_{X}\Mb_X^{-1}\D F_X^*$ follows from the su\-rjec\-ti\-vi\-ty of  $\D F_{X}$ proved in~\cite{GSAS21} and the injectivity of $\D F_X^*$ which can be verified by straightforward calculations. 

Before moving on, we introduce a~Lyapunov operator which is crucial for the development of geometric ingredients on $\Spkn$. Given $X\in\Spkn$, we define
\begin{equation}\label{eq:Lyap_operator}
\arraycolsep=2pt
\begin{array}{rccl}
\mathrm{Lyap}_{X,\Mb_X}
 \; :\;\;&\mathbb{R}^{2k\times 2k}&\longrightarrow &\mathbb{R}^{2k\times 2k}\\ 
	&\varOmega&\longmapsto& \big(X^T\!J_{2n}^T\Mb_X^{-1}J_{2n}^{}X\big)\,\varOmega + \varOmega \,\big(X^T\!J_{2n}^T\Mb_X^{-1}J_{2n}^{}X\big),  
\end{array}
\end{equation}
where the spd matrix $\Mb_X$ stems from the tractable metric. Note that the coefficient matrix $X^T\!J_{2n}^T\Mb_X^{-1}J_{2n}^{}X$ is symmetric and positive definite. Hence, it follows from~\cite[Thm.~4.4.6]{HornJ91} that the Lyapunov operator $\mathrm{Lyap}_{X,\Mb_X}$ is invertible. 

The following proposition provides a~matrix expression for the projection $\proj$ without invoking derivatives.

\begin{proposition}\label{prop:projPX}
Given $X\in\Spkn$ and $Y\in\R^{2n\times 2k}$, the orthogonal projection of~$Y$ onto $\TX$  with respect to the tractable metric $g_{\Mb_X}\!$ has the form
\begin{equation}\label{eq:projMX}
	\proj(Y) = Y-\Mb_X^{-1}J_{2n}X \varOmega_{X,Y}, 
\end{equation}
where $\varOmega_{X,Y}\in\skewset(2k)$ is the solution to the Lyapunov equation
\begin{equation}\label{eq:Lyap} 
	\mathrm{Lyap}_{X,\Mb_X}(\varOmega)=2\,\skew(X^T\!J_{2n}^TY).
\end{equation}
\end{proposition}
\begin{proof}
First, it follows from \eqref{eq:DF} that
\[
\D F_{X}(Y) = -(X^T\!J_{2n}^TY - Y^TJ_{2n}^{}X) = -2\,\skew(X^T\! J_{2n}^TY).
\]	
Further, using 	\eqref{eq:DF} and \eqref{eq:DF_star_detail}, we obtain that for all $\varOmega\in\skewset(2k)$,
\begin{align*}
	\D F_{X}\Mb_{X}^{-1}\D F_X^*(\varOmega) &= 2\,\big(X^T\!J_{2n}^T\Mb_{X}^{-1}J_{2n}^{}X\,\varOmega +\varOmega\, X^T\!J_{2n}^T\Mb_{X}^{-1} J_{2n}^{}X\big).
\end{align*}
Using the above expressions, the equation 
$\D F_{X}\Mb_{X}^{-1}\D F_X^*(\varOmega)=\D F_{X}(Y)$ 
is equi\-valent to the Lyapunov equation \eqref{eq:Lyap} via the substitution $\varOmega_{X,Y} :=-2\,\varOmega$. Moreover, we have
$\varOmega=(\D F_{X}\Mb_{X}^{-1}\D F_X^*)^{-1}\D F_{X}(Y)$. Finally, it follows from expressions~\eqref{eq:Du_orth_proj} and~\eqref{eq:DF_star_detail} that
\begin{align*}
    \proj(Y) &= Y - \Mb_X^{-1} \D F_X^* (\D F_{X}\Mb_{X}^{-1}\D F_X^*)^{-1}\D F_{X}(Y) \\
    &= Y -\Mb_X^{-1}\D F_X^*(\varOmega) \\
    &=  Y +2\,\Mb_X^{-1} J_{2n}^{}X\varOmega \\ 
    &=Y -\Mb_X^{-1} J_{2n}^{}X\varOmega_{X,Y}. 
\end{align*}
This completes the proof. \qed
\end{proof}

It is worth to note that the Lyapunov equation \eqref{eq:Lyap} has a~unique skew-symmetric solution $\varOmega_{X,Y}$ since the Lyapunov operator $\mathrm{Lyap}_{X,\Mb_X}$ is invertible and the right-hand side is skew-symmetric. The subscript~$Y$ in $\varOmega_{X,Y}$ is used to emphasize the dependence of the solution on $Y$ involved in the right-hand side.

Given $X\in\Spkn$, the \emph{Riemannian gradient} of a~continuously differentiable function \mbox{$f:\Spkn\rightarrow \R$} at $X$ with respect to a~tractable metric $g_{\Mb_X}$, denoted by $\rgrad{X}$, is defined as the unique element of $\TX$ which satisfies the condition
\[
g_{\Mb_X}(\rgrad{X} ,Z)= \langle\nabla\bar{f}(X), Z\rangle \qquad \text{for all }Z\in\TX.
\]
Based on \cite[Prop.~3.2]{Ngu23} or \cite[Eq.~(3.20)]{ShuA23}, the Riemannian gradient  can be represented as 
\begin{equation}\label{eq:Du_Rgradient}
\rgrad{X} = \proj\big(\Mb_X^{-1}\nabla\bar{f}(X)\big).
\end{equation} 

Given an spd matrix $\Mb_X$, we summarize the basic geometric ingredients for the symplectic Stiefel manifold $\Spkn$ in Table~\ref{tab:notion_summary} and give a geometric illustration in Figure~\ref{fig:illustration}. Notice that the tractable metric $g_{\Mb_X}$ along with the matrix $\Mb_X$ plays a~crucial role in the Riemannian geometry on $\Spkn$. Moreover, we observe from Figure~\ref{fig:illustration} that it is possible to consider preconditioning for the optimization problem~\eqref{eq:opt_prob} via a~specific metric.

\begin{figure}[tpb]
    \centering
    \includegraphics[width=.45\textwidth]{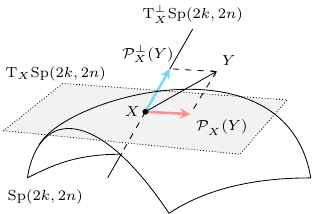}\qquad\quad
    \includegraphics[width=.45\textwidth]{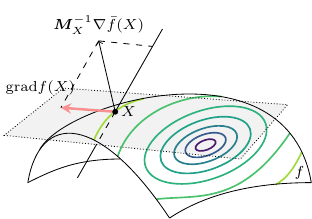}
    \caption{Geometric illustration of Riemannian ingredients on the symplectic Stiefel manifold\label{fig:illustration}}
\end{figure}

\begin{table}[tbp]
	\centering
	\caption{Geometric ingredients on the symplectic Stiefel manifold $\Spkn$ with respect to the metric 
 $g_{\Mb_X}$;
 $\varOmega_{X,Y}\in\skewset(2k)$ is the solution to the Lyapunov equation~\eqref{eq:Lyap}.}
	\label{tab:notion_summary}
	\begin{tabular}{lll}
		\toprule
		{\sc~} & {Notation} & {Expression} \\\midrule
		{metric} & {$g_{\Mb_X}(Z_1,Z_2)$}   & {$\langle Z_1,\Mb_{X}Z_2\rangle$}
            \\\cmidrule[.1pt](r){2-3}
		{tangent space}~~ & $\TX$  & {$XJ_{2k}W + J_{2n}X_{\perp} K, W\!\in\! \symset(2k), K\!\in \!\R^{(2n-2k)\times 2k}$} 
		\\\cmidrule[.1pt](r){2-3}
		{normal space}~~ & $\NX$  & {$\Mb_{X}^{-1}J_{2n}X\varOmega, \; \varOmega\in\skewset(2k)$} 
		\\\cmidrule[.1pt](r){2-3}
		\multirow{2}{*}{projection} & {$\proj(Y)$}    & $Y-\Mb_X^{-1}J_{2n}X\, \varOmega_{X,Y}$  \\
		& $\projn(Y)$  & $\Mb_X^{-1}J_{2n}X\, \varOmega_{X,Y}$
            \\\cmidrule[.1pt](r){2-3}
		{gradient} & $\grad f(X)$  & $\proj\big(\Mb_X^{-1}\nabla\bar{f}(X)\big)$  \\
		\bottomrule
	\end{tabular}
\end{table}

\begin{remark}[Preconditioning by metrics]
In various applications, preconditioned optimization methods show a~prominent acceleration to vanilla gradient descent methods, e.g., linear systems~\textup{\cite{KressnSV2016}}, matrix completion~\textup{\cite{BoumA15}}, and tensor completion~\textup{\cite{KasM2016}}. The trick is to come up with a~delicate metric that exploits the second-order information of objectives, reflecting a~preconditioning effect. In the same spirit, the authors in~\textup{\cite{AltPS23,GaoPY2023,ShuA23}} provide a~practical way to construct such a~metric for specific problems. It is worth to note that the new family of metrics defined in~\eqref{eq:metricM} benefits these advantages and allows us to fit the preconditioning framework. The Riemannian gradient
$\rgrad{X} = \proj\left(\Mb_X^{-1}\nabla\bar{f}(X)\right)$ can then be utilized to develop a~preconditioned RGD method through an~appropriate matrix~$\Mb_X$. 
\end{remark}

For some specific choices of the matrix $\Mb_X$, we obtain well-known 
metrics on the symplectic Stiefel manifold $\Spkn$ such as the canonical-like metric and the Euclidean metric studied in \cite{GSAS21a,GSAS21,GSS24}.
In the following subsections, we consider these two metrics in more detail and check how they are recovered by tailoring a~specific matrix $\Mb_X$. The recovered results are important to develop Riemannian Hessians in the next section.

%-----------------------------------------------------------------------------%
%
\subsection{Canonical-like metric}
For a~parameter $\rho>0$ and the tangent vectors $Z_i=XJ_{2k}W_i+J_{2n}X_\perp K_i$ with \mbox{$W_i\in\symset(2k)$} and $K_i\in\R^{(2n-2k)\times 2k}$ for $i=1,2$, the \emph{canonical-like metric} is defined as 
\begin{equation}\label{eq:canonical_metric}
g_{\Mb_{X,c,\rho}}(Z_1,Z_2):= \frac{1}{\rho}  \tr(W_1^T W_2^{})+\tr(K_1^T K_2^{}) =\tr(Z_1^T\Mb_{X,c,\rho}Z_2^{}),
\end{equation}
where the spd matrix $\Mb_{X,c,\rho}$ is given by 
\begin{align}\label{eq:B_X2}
\Mb_{X,c,\rho}
&= \frac{1}{\rho}J_{2n}^{}XX^TJ_{2n}^T - X_\perp^{}(X_\perp^TJ_{2n}X_\perp^{})^{-2}X_\perp^T.
\end{align}
Note that imposing  an additional orthonormalization condition on $X_\perp$, for example,
$ X_\perp^TX_\perp^{}=I_{2n-2k}$ or 
\begin{equation}\label{eq:choice_Xperp}
\big(X_\perp(X_\perp^TJ_{2n}X_\perp)^{-1}\big)^T\big(X_\perp^{}(X_\perp^TJ_{2n}X_\perp)^{-1}\big)=I_{2n-2k},
\end{equation}
yields that $\Mb_{X,c,\rho}$ (and therefore the metric $g_{\Mb_{X,c,\rho}}$) is independent of $X_\perp$ and, hence, $\Mb_{X,c,\rho}$ varies smoothly with~$X$, see \cite[Prop.~4.1]{GSAS21}.  
Computing the inverse
\begin{equation}\label{eq:B_X_E_inverse}
\Mb_{X,c,\rho}^{-1} = \rho\, XX^T + J_{2n}X_\perp^{}X_\perp^TJ_{2n}^T
\end{equation}
and $\Mb_{X,c,\rho}^{-1}J_{2n}^{}X=\rho\,XJ_{2k}^{}$, we derive from \eqref{eq:normal}  the normal space to $\Spkn$ with respect to the canonical-like metric
\begin{equation}\label{eq:normal_space_canonical}
\NXc 
= \bigl\{ XJ_{2k}\,\varOmega  \enskip: \enskip \varOmega\in\skewset(2k)\bigr\}.
\end{equation}
Furthermore, the Lyapunov equation \eqref{eq:Lyap} with $\Mb_{X}$ replaced by $\Mb_{X,c,\rho}$ has the solution $\varOmega_{X,Y}=\frac{1}{\rho}\skew(X^T\!J_{2n}^TY)$. Inserting it into~\eqref{eq:projMX}, we obtain the orthogonal projection 
\begin{equation}\label{eq:orth_proj_canon}
\projc(Y) =  Y-XJ_{2k}\,\skew(X^T\!J_{2n}^TY), \qquad Y\in\Rbnk,
\end{equation}
which is, notably, independent of both $\rho$ and the choice of $X_\perp$. Note that the representation \eqref{eq:orth_proj_canon} is equivalent to that obtained in \cite[Prop.~4.3]{GSAS21}.
The Riemannian gradient of $f$ with respect to the canonical-like metric $g_{\Mb_{X,c,\rho}}$
can then be determined from \eqref{eq:Du_Rgradient} and \eqref{eq:orth_proj_canon} as
\begin{align}
\rgradc{X} & =  \projc\big(\Mb_{X,c,\rho}^{-1}\nabla\bar{f}(X)\big)\nonumber\\
& = \big(\rho\, XX^T + J_{2n}X_\perp^{}X_\perp^TJ_{2n}^T\big)\nabla\bar{f}(X)
-\rho\,XJ_{2k}\,\skew\big(J_{2k}^TX^T\nabla\bar{f}(X)\big)\nonumber\\
& =\rho\, XJ_{2k}\,\sym\big(J_{2k}^TX^T\nabla\bar{f}(X)\big)+ J_{2n}X_\perp^{}X_\perp^{T}J_{2n}^T\nabla \bar{f}(X),\label{eq:Rgrad_canon}
\end{align}
cf. \cite[Prop.~4.5]{GSAS21}. 

%------------------------------------------------------------------------%
%
\subsection{Weighted Euclidean metric}
\label{ssec:euclM}
Setting $\Mb_{X} = \Mb$ to be a~constant spd matrix, we obtain the \emph{weighted Euclidean metric}
\begin{equation}\label{eq:metric_euclM}
g_{\Mb}(Z_1, Z_2) :=  \langle Z_1, \Mb\, Z_2\rangle,  \qquad  Z_1,Z_2 \in \TX.
\end{equation}
Then \eqref{eq:normal} and \eqref{eq:projMX} imply that the associated normal space to $\Spkn$ at $X$ is characterized by 
\begin{equation}\label{eq:normalspaceM}
\NXM = \big\{\Mb^{-1}J_{2n}X\,\varOmega\ :\ \varOmega \in \skewset(2k)\big\},
\end{equation}
and the orthogonal projection onto the tangent space $\TX$ takes the form 
\begin{align}\label{eq:orth_proj_EuclidM}
\projM(Y) = Y - \Mb^{-1}J_{2n}X\varOmega_{X,Y},\qquad 	Y\in \Rbnk,
\end{align}
where $\varOmega_{X,Y}$ is the skew-symmetric solution to the Lyapunov equation 
\begin{equation}\label{eq:Lyapeq}
\mathrm{Lyap}_{X,\Mb}(\Omega)=2\,\skew(X^T\!J_{2n}^TY).
% X^TJ_{2n}^T\Mb^{-1}J_{2n}^{}X\,\varOmega_{X,Y} + \varOmega_{X,Y}\, X^TJ_{2n}^T\Mb^{-1}J_{2n}^{}X = 2\,\skew(X^T\!J_{2n}^TY).
\end{equation}
Finally, it follows from \eqref{eq:Du_Rgradient} and \eqref{eq:orth_proj_EuclidM} that the Riemannian gradient of~$f$ with respect to the weighted Euclidean metric $g_\Mb$ is given by
\begin{equation}\label{eq:grad_EuclidM}
\rgradM{X} = \Mb^{-1}\nabla \bar{f}(X)-\Mb^{-1}J_{2n}X\varOmega_{X,\Mb^{-1}\nabla \bar{f}},
\end{equation}
where $\varOmega_{X,\Mb^{-1}\nabla \bar{f}}$ solves the Lyapunov equation 
\eqref{eq:Lyapeq} with $Y=\Mb^{-1}\nabla \bar{f}(X)$. 

Replacing $\Mb$ by the identity matrix in this subsection, we obtain all formulas established for the Euclidean metric in \cite{GSAS21a}. In this case, we will use the subscript~$e$ instead of $\Mb$ to mention objects in the Euclidean metric, e.g., $g_e(\cdot,\cdot), \proje(\cdot)$, and $\rgrade{X}$.

%------------------------------------------------------------------------------%
%
\section{Riemannian Hessians}
\label{sec:Rhessian}

On linear spaces, Newton's method for solving optimization problems requires the second-order derivative of the cost function, which is the directional derivative of its gradient. This concept can be extended to Riemannian manifolds by using a~Riemannian connection, e.g., \cite[Sect.~5.3]{AbsiMS08}. In this section, employing the operator-valued framework from \cite{Ngu23} for computing geometric quantities including Riemannian Hessian in the case of {tractable} metrics, we compute the Riemannian Hessian of a~smooth function on the symplectic manifold $\Spkn$.

Given a $C^2$-function $f$ defined on $\Spkn$, the \emph{Riemannian Hessian} of $f$ at \mbox{$X\in\Spkn$}, denoted by $\hess f(X)$, is defined as a~linear operator on the tangent space $\TX$ given by
\[
\Hess{X}{Z} = \nabla_{\!Z}\,\grad f(X), \qquad Z\in \TX,
\]
where $\nabla$ denotes the Riemannian connection on $\Spkn$. 
For computing the Riemannian Hessian on $\Spkn$ endowed with the tractable metric $g_{\Mb_X}\!$ with an~spd matrix~$\Mb_{X}$, we use the framework \cite[Thm.~3.1]{Ngu23} which (adapted to our notation) can be stated as follows. 

\begin{theorem}[Riemannian Hessian]
Given $X\in\Spkn$, the Riemannian Hessian of a~$C^2$-function 
\mbox{$f:\Spkn\longrightarrow\R$} at~$X$ with respect to the tractable metric~$g_{\Mb_X}\!$ defined in  \eqref{eq:metricM} is given by
\begin{align}
\Hess{X}{Z} 
  & = \proj\big(\D_Z\,\rgrad X+\Mb_X^{-1}\mathcal{K}(Z,\rgrad X)\big) \notag\\
  & = \proj\big(\Mb_X^{-1}\Nablaf{X}{Z} + \D_Z^{} \proj(\Mb_X^{-1}\nabla \bar{f}(X))\label{eq:Rhessian}\\ 
		&\quad - \Mb_X^{-1} \D_Z^{} \Mb_X^{}(\Mb_X^{-1}\nabla \bar{f}(X)) 
  + \Mb_X^{-1}\mathcal{K}\big(Z, \proj(\Mb_X^{-1}\nabla \bar{f}(X))\big)\big), \notag
\end{align}
where $\proj$ is the orthogonal projection onto $\TX$ as in \eqref{eq:projMX}, 
the mapping $\mathcal{K}$ is defined as
\begin{equation}\label{eq:K}
	\arraycolsep=2pt
	\begin{array}{rccl}
		\mathcal{K} :&\TX\times \TX&\longrightarrow &\Rbnk\\
		&(Z,U)&\longmapsto&\frac{1}{2}\bigl(\D_Z^{} \Mb_X(U) + \D_U^{} \Mb_X(Z)-\calX(Z,U)\bigr),
	\end{array}
\end{equation}
with the mapping $\calX: \TX\times \TX\longrightarrow \TX$ satisfying the condition
\begin{equation}\label{eq:Chi}
	\big\langle\calX(Z,U),V\big\rangle=\big\langle Z,\D_{V}^{}\Mb_X(U)\big\rangle\quad\text{for all } 
	Z, U, V\in\TX.
\end{equation}
\end{theorem}

At first glance, the Riemannian Hessian in~\eqref{eq:Rhessian} seems a bit intimidating in its specific form. Nevertheless, it is the first time that we figure out the second-order geometry of $\Spkn$ in terms of the tractable metric. Note that the development of~\eqref{eq:Rhessian} firmly builds upon the metric and the derivatives (e.g., $\D_Z^{} \proj$ and $\D_Z^{} \Mb_X$). Hence, the computation of the Riemannian Hessian~\eqref{eq:Rhessian} is available but not straightforward. In the following subsections, we derive its explicit formulation in matrix form with respect to the metrics that enjoy known geometric results, such as the canonical-like and weighted Euclidean metrics.

%--------------------------------------------------------------------------------%
%
\subsection{Riemannian Hessian with respect to the canonical-like metric}
\label{Sec:hessian_canonical}

First, we consider the canonical-like metric $g_{\Mb_{X,c,\rho}}\!$ 
with the spd matrix $\Mb_{X,c,\rho}$ given in~\eqref{eq:B_X2}. 
We begin with establishing some directional derivatives required for constructing the Rieman\-nian Hessian. 

\begin{lemma}\label{lem:Diff_orth_proj_canonical}
Let $X\in\Spkn$, $Z\in\TX$, $Y\in\Rbnk$, and let the orthogonal projection $\projc$ be as in \eqref{eq:orth_proj_canon}. Then the action of the directional derivative of~$\projc$ at $X$ in the direction $Z$ on $Y$ is given by
\begin{equation}\label{eq:Diff_orth_proj_canonical}
	\D_Z^{}\projc(Y) = -XJ_{2k}\,\skew(Z^TJ_{2n}^TY) - Z J_{2k}\,\skew(X^TJ_{2n}^TY).
\end{equation}
\end{lemma}
\begin{proof} 
Let $ \varUpsilon(t)\subset\Spkn$ be a~smooth curve defined on a~neighborhood of~$t=0$ such that $ \varUpsilon(0) = X$ and $\dot{ \varUpsilon}(0) = Z$.  It follows from \eqref{eq:orth_proj_canon} that
\begin{align*}
	\calP_{ \varUpsilon(t),c}(Y) - \projc (Y) &= -\bigl( \varUpsilon(t)J_{2k}\,\skew( \varUpsilon(t)^TJ_{2n}^TY) -  \varUpsilon(t)J_{2k}\,\skew(X^TJ_{2n}^TY)\bigr)\\
	&\quad - \bigl( \varUpsilon(t)J_{2k}\,\skew(X^TJ_{2n}^TY) - XJ_{2k}\,\skew(X^TJ_{2n}^TY)\bigr).
\end{align*}
Then, dividing both sides of this equality by $t$ and letting $t$ tend to zero, we obtain~\eqref{eq:Diff_orth_proj_canonical}. \qed
\end{proof}

Next, we calculate the directional derivative of the matrix $\Mb_{X,c,\rho}$ in~\eqref{eq:B_X2} that determines the canonical-like metric. For the sake of brevity, we impose the condition~\eqref{eq:choice_Xperp} on $X_\perp$.
In this case, as shown in \cite[Prop.~4.1]{GSAS21}, we obtain that 
\begin{align*}
-X_\perp\big(X_\perp^T J_{2n}X_\perp\big)^{-2}X_\perp^T 
%&= \Bigl(X_\perp\left(X_\perp^T J_{2n}X_\perp\right)^{-1}\Bigr)\Bigl(\left(X_\perp^T J_{2n}X_\perp\right)^{-1}X_\perp^T\Bigr)\\
= I_{2n}-X(X^TX)^{-1}X^T
\end{align*}
is the orthogonal projection onto the orthogonal complement of the subspace spanned by the columns of $X$, which is denoted by $\PXp:= I_{2n} - X(X^TX)^{-1}X^T$. Then $\Mb_{X,c,\rho}$ in~\eqref{eq:B_X2} can be written as 
\begin{equation}\label{eq:B_X3}
\Mb_{X,c,\rho} = \frac{1}{\rho}J_{2n}^{}XX^TJ_{2n}^T + \PXp.
\end{equation}
This shows that $\Mb_{X,c,\rho}$ is indeed independent of $X_\perp$.
\begin{lemma}\label{lem:Diff_B_X}
Let $X\in\Spkn$ and $Z\in\TX$. The directional derivative of $\Mb_{X,c,\rho}$ in \eqref{eq:B_X3}, considered as an~operator on $\Rbnk$, at~$X$ in the direction~$Z$ is given by
\begin{align}\label{eq:Diff_B_X}
		\D_Z^{} \Mb_{X,c,\rho} &=2\,\sym\Bigl(\frac{1}{\rho}J_{2n}XZ^TJ_{2n}^T 
                        - \PXp Z(X^TX)^{-1}X^T\Bigr).
\end{align}
\end{lemma}
\begin{proof}
Just as in the proof above, we assume that $ \varUpsilon(t)\subset\Spkn$ is a~smooth curve defined on a~neighborhood of~$t=0$ such that $ \varUpsilon(0) = X$ and $\dot{ \varUpsilon}(0) = Z$. First, by differentiating the relation $\big( \varUpsilon(t)^T \varUpsilon(t)\big)^{-1}\big( \varUpsilon(t)^T \varUpsilon(t)\big) = I_{2k}$ at $t=0$, we obtain that
\begin{equation}\label{proof:Diff_B_X}
	\left.\frac{\d}{\d t }\big( \varUpsilon(t)^T \varUpsilon(t)\big)^{-1}\right|_{t=0} = -(X^TX)^{-1}(X^TZ + Z^TX)(X^TX)^{-1}.
\end{equation}
Using \eqref{eq:B_X3}, we find that
\begin{align*}
	\Mb_{ \varUpsilon(t),c,\rho}-\Mb_{X,c,\rho} 
    &= \frac{1}{\rho}J_{2n}\Bigl( \varUpsilon(t) \varUpsilon(t)^T -  \varUpsilon(t)X^T +  \varUpsilon(t)X^T - XX^T\Bigr)J_{2n}^T\\ 
	&\quad - \varUpsilon(t)\big( \varUpsilon(t)^T \varUpsilon(t)\big)^{-1}( \varUpsilon(t)-X)^T \\
    &\quad - ( \varUpsilon(t)- X)\big( \varUpsilon(t)^T \varUpsilon(t)\big)^{-1}X^T \\ 
	&\quad - X\Bigl(\big( \varUpsilon(t)^T \varUpsilon(t)\big)^{-1}- (X^TX)^{-1}\Bigl)X^T .
\end{align*}
Then, dividing both sides by $t$ and letting $t$ tend to zero with \eqref{proof:Diff_B_X} in mind, \eqref{eq:Diff_B_X} is obtained. \qed
\end{proof}

Having computed $\D_Z^{} \Mb_{X,c,\rho}$, we determine now the mapping $\calX$ defined in~\eqref{eq:Chi}.
\begin{lemma}\label{lem:index_raising}
Let $X\in\Spkn$ and $Z, U\in\TX$. The mapping $\cal X$ satisfying~\eqref{eq:Chi} is given by
\begin{align}\label{eq:index_raising}
	\calX(Z, U) &= \frac{2}{\rho}\,J_{2n}^T\,\sym(U Z^T)J_{2n}X 
 - 2\,\PXp\, \sym(U Z^T)X(X^TX)^{-1}.
\end{align}
\end{lemma}
\begin{proof}
For any $Z,U,V\in\TX$, using \eqref{eq:Diff_B_X} and properties of the trace function, we obtain that
\begin{align*}
	\big\langle Z,\D_{V}^{}\Mb_{X,c,\rho}( U)\big\rangle\! &= \tr\Bigl(\frac{1}{\rho}Z^TJ_{2n}(XV^T + VX^T)J_{2n}^T U \\ &\quad - Z^TX(X^TX)^{-1}V^T\PXp U- Z^T\PXp V(X^TX)^{-1}X^T U\Bigr)\\ 
	&=\tr\Bigl(V^T\!\big(\frac{2}{\rho}J_{2n}^T\sym(U Z^T)J_{2n}^{}X\! -2\PXp\sym(U Z^T)X(X^T\!X)^{-1} 
	\big)\Bigr).
\end{align*}
This immediately implies  \eqref{eq:index_raising}. \qed
\end{proof}

Next, we derive an expression for $\mathcal{K}$ in \eqref{eq:K}.
\begin{lemma}\label{lem:compute_K}
Let $X\in\Spkn$ and $Z, U\in\TX$. Then the mapping $\mathcal{K}$ defined in \eqref{eq:K} is given by
\begin{align}
	\mathcal{K}(Z, U) &= \frac{1}{\rho}J_{2n}\bigl(X\skew(Z^TJ_{2n}^T U) + 2\,\sym(Z U^T)J_{2n}X\bigr) \nonumber\\
	& \quad - \PXp U\,\skew\big((X^TX)^{-1}X^TZ\big) - \PXp Z\,\skew\big((X^TX)^{-1}X^T U\big) \nonumber\\
	& \quad - X(X^TX)^{-1}\sym\big(Z^T\PXp U\big). \label{eq:K_detail}
\end{align}
\end{lemma}
\begin{proof}
The expression \eqref{eq:K_detail} is obtained simply by concatenating the terms in the definition~\eqref{eq:K} using \eqref{eq:Diff_B_X}, \eqref{eq:index_raising}, and the tangent condition \eqref{eq:tangentspace_1}. Indeed,
\begin{align} \notag
	2\,\mathcal{K}(Z, U) &= \D_Z^{} \Mb_{X,c,\rho}(U) + \D_U^{} \Mb_{X,c,\rho}(Z) - \calX(Z, U)\\ \notag
	&= \frac{1}{\rho}\Bigl(J_{2n}^{}(XZ^T+Z X^T)J_{2n}^T U + J_{2n}^{}(X U^T+U X^T)J_{2n}^TZ \\ \notag
	&\quad + J_{2n}\bigl( U Z^T\! + Z U^T\bigr)J_{2n}X \Bigr)
	+\PXp U Z^TX(X^TX)^{-1} \\ \notag
	&\quad +\PXp Z U^TX(X^TX)^{-1} - \PXp  Z(X^TX)^{-1}X^T U \\ \notag
	&\quad - \PXp U(X^TX)^{-1}X^T Z - X(X^TX)^{-1}\bigl( Z^T\PXp U +  U^T\PXp Z\bigr)\\\notag
	&= \frac{1}{\rho}\Bigl(J_{2n}^{}X\bigl( Z^TJ_{2n}^T U +  U^TJ_{2n}^T Z\bigr) + 2J_{2n}^{}\bigl( U Z^T +  Z U^T\bigr)J_{2n}^{}X\Bigr)\\ \notag
	&\quad - 2\,\PXp U\skew\big((X^TX)^{-1}X^T Z\big) - 2\,\PXp Z\skew\big((X^TX)^{-1}X^T U\big) \\
	&\quad - 2\,X(X^TX)^{-1}\sym\big( Z^T\PXp U\big)  .\notag
\end{align} 
This completes the proof.\qed
\end{proof}

Finally, we achieve an expression for the Riemannian Hessian $\hess_cf(X)$. 
\begin{theorem}\label{th:Rhess_canon-like}
Given a $C^2$-function $f:\Spkn\longrightarrow\R$ with its ambient gradient and Hessian $\nabla\bar{f}(X)$ and $\nabla^2 \bar{f}(X)$, respectively.
The Riemannian Hessian of $f$ at \mbox{$X\in\Spkn$} with respect to the canonical-like metric
applied to $Z\in\TX$ is determined as 
\begin{align*}
	\Hessc{X}{Z} 
       & = \projc\Bigl(\Mb_{X,c,\rho}^{-1}\Nablaf{X}{Z} 
        - \rho\,  Z J_{2k}\skew\bigl(J_{2k}^TX^T\nabla \bar{f}(X)\bigr) \\
	& \quad +2\,\sym\bigl(\rho\,XZ^T
	-2\,\skew\bigl(XJ_{2k}^{}Z^T\bigr)J_{2n}^TP_X^T\bigr)\nabla\bar{f}(X)\\
	& \quad + \rho\, X\sym\bigl(X^TJ_{2n} Z\nabla\bar{f}(X)^TXJ_{2k} 
	-Z^T\nabla\bar{f}(X)\bigr)\Bigr)\\
	\notag
	& \quad +P_X\Bigl( P_X^TJ_{2n}^{}\bigl( Z\,\sym\bigl(\nabla\bar{f}(X)^TXJ_{2k}\bigr) + \frac{1}{\rho}P_X^T\nabla\bar{f}(X) Z^TJ_{2n}^{}X\bigr)\\ \notag
	& \quad - J_{2n}^{}XJ_{2k}^{}\,\sym\bigl( Z^TP_X^T\nabla\bar{f}(X)\bigr)\\ \notag
	& \quad - Z\,\skew\bigl(J_{2k}^{}X^TJ_{2n}^{}P_X^T\nabla\bar{f}(X) + \rho\, J_{2k}\sym\bigl(J_{2k}^TX^T\nabla\bar{f}(X)\bigr)\bigr)\\ \notag
	& \quad -P_X^T\nabla\bar{f}(X)\skew\bigl((X^TX)^{-1}X^T Z\bigr)\Bigr),
\end{align*}
where  
\begin{equation}\label{eq:prPXperp}
P_X=I_{2n}-XJ_{2k}^{}X^TJ_{2n}^T
\end{equation}
is an~oblique projection onto the tangent space $\TX$.
\end{theorem}

The proof is lengthy and therefore moved to Appendix~\ref{appendix:proof_Prop_Rhess_canon-like}. Theorem~\ref{th:Rhess_canon-like} reveals that the Riemannian Hessian with respect to the canonical-like metric is indeed given explicitly in matrix form, which only involves matrix multiplications and inverses, and does not require solving a~matrix equation.

\begin{remark}
When the symplectic Stiefel manifold  $\Spkn$ reduces to the symplectic group $\Spn$, i.e., the special case $k=n$, we have $P_X=0$. Then the expression for the Riemannian Hessian in Theorem~\ref{th:Rhess_canon-like} is simplified to
\begin{align*}
	&\Hessc{X}{Z} =\projc\Bigl(
	\Mb_{X,c,\rho}^{-1}\Nablaf{X}{Z} - \rho\,  Z J_{2k}\skew\bigl(J_{2k}^TX^T\nabla \bar{f}(X)\bigr) \\
	& \qquad +2\,\rho\,\sym\bigl(XZ^T\bigr)\nabla\bar{f}(X)
	+ \rho\, X\sym\bigl(X^TJ_{2n} Z\nabla\bar{f}(X)^TXJ_{2k} 
	-Z^T\nabla\bar{f}(X)\bigr)\Bigr).
\end{align*}
\end{remark}

%---------------------------------------------------------------------%
%
\subsection{Riemannian Hessian with respect to the weighted Euclidean metric}
\label{ssec:Rhessian_EuclideanM}
The Riemannian Hessian with respect to the weighted Euclidean metric $g_{\Mb}$ defined in~\eqref{eq:metric_euclM} 
is simpler than that derived for the canonical-like metric. Indeed, due to \mbox{$\Mb_{X}=\Mb$}, all terms in~\eqref{eq:Rhessian} containing the derivative of this matrix vanish. As a~consequence, we obtain that
\begin{equation}\label{eq:Rhessian_M}
\HessM{X}{Z} = \projM\bigl(\Mb^{-1}\Nablaf{X}{Z} + \D_{Z}^{}\projM(\Mb^{-1}\nabla\bar{f}(X))\bigr).
\end{equation}
Note that when $\Mb=I_{2n}$, the metric reduces to the Euclidean metric, and the above formulation is simplified to 
\[
\Hesse{X}{Z} = \proje\bigl(\Nablaf{X}{Z} + \D_{Z}^{}\proje(\nabla\bar{f}(X))\bigr),
\]
which coincides with the known result in \cite[(7)]{AbsMT13} for a~general Riemannian submanifold.

To get a detailed expression for the Riemannian Hessian $\hess_\Mb f(X)$, we need the directional derivative of the projection $\projM$. 
\begin{lemma}\label{lem:directionalDerivM}
Given $X\in \Spkn$, $Z\in\TX$, $Y\in \Rbnk$, and the orthogonal projection $\projM$ as in \eqref{eq:orth_proj_EuclidM}. Then 
the action of the directional derivative of~$\projM$ at $X$ in the direction $Z$  on $Y$ is computed as
\begin{equation}\label{eq:directionalDerivM}
	\D_{Z}\projM(Y) = -\Mb^{-1}J_{2n}( Z\varOmega_{X,Y}+X\varXi_{X,Y, Z}),
\end{equation}
where $\varOmega_{X,Y}$ solves the Lyapunov equation \eqref{eq:Lyapeq} and $\varXi_{X,Y,Z}$ is the solution to the Lyapunov equation
\begin{equation}\label{eq:eqdirectionalDerivM}
    \mathrm{Lyap}_{X,\Mb}(\varXi) = 2\,\skew \big( Z^TJ^T_{2n}Y + 2\,\varOmega_{X,Y}^T\sym(X^TJ_{2n}^T\Mb^{-1}J_{2n}^{} Z)\big).
\end{equation}
\end{lemma}
\begin{proof}
Let $\varUpsilon(t) \subset \Spkn$ be a~smooth curve defined on a~neighborhood of~$t=0$ such that $ \varUpsilon(0) = X$ and $\dot{ \varUpsilon}(0) =  Z$. By definition, it holds that
\begin{align}\notag
	\D_{Z}^{}\projM(Y) &= \lim\limits_{t\to 0}\frac{1}{t}\big(\calP_{ \varUpsilon(t),\Mb}(Y) - \calP_{ \varUpsilon(0),\Mb}(Y)\big)\\ \notag
	&= \lim\limits_{t\to 0}\frac{1}{t}(Y - \Mb^{-1}J_{2n} \varUpsilon(t)\varOmega_{ \varUpsilon(t),Y} - Y + \Mb^{-1} J_{2n}X\varOmega_{X,Y})\\ 
%	&= -\Mb^{-1}J_{2n}\lim\limits_{t\to 0}\frac{1}{t}( \varUpsilon(t)\varOmega_{ \varUpsilon(t),Y} 
%      - X\varOmega_{ \varUpsilon(t),Y} + X \varOmega_{ \varUpsilon(t),Y} - X\varOmega_{X,Y})\\
	&=-\Mb^{-1}J_{2n}\lim\limits_{t\to 0}\frac{1}{t}\big(( \varUpsilon(t) -  X) \varOmega_{ \varUpsilon(t),Y} + X( \varOmega_{ \varUpsilon(t),Y}- \varOmega_{X,Y})\big).
	\label{eq:directionalDeriv_proof1M}
\end{align}
Further, subtracting equation \eqref{eq:Lyapeq} with $\varOmega=\varOmega_{X,Y}$ from the one that determines $ \varOmega_{ \varUpsilon(t),Y}$, i.e.,
\begin{equation*}
    \mathrm{Lyap}_{\varUpsilon(t),\Mb}(\varOmega_{\varUpsilon(t),Y})= 2\,\skew( \varUpsilon(t)^TJ_{2n}^TY),
\end{equation*}
we obtain that
\begin{align}\notag
	&2\,\skew\big((\varUpsilon(t)-X)^TJ_{2n}^TY\big) \\\notag
	&\quad 
    = \varUpsilon(t)^T J_{2n}^T\Mb^{-1}J_{2n}^{}\varUpsilon(t) \varOmega_{ \varUpsilon(t),Y} 
    - X^TJ_{2n}^T\Mb^{-1}J_{2n}^{}X \varOmega_{X,Y}\\\notag 
    & \qquad + \varOmega_{ \varUpsilon(t),Y} \varUpsilon(t)^T J_{2n}^T\Mb^{-1}J_{2n}^{} \varUpsilon(t) 
    - \varOmega_{X,Y}X^TJ_{2n}^T\Mb^{-1}J_{2n}^{}X\\\notag
	&\quad =\big(( \varUpsilon(t) - X)^TJ_{2n}^T\Mb^{-1}J_{2n}^{} \varUpsilon(t) +X^TJ_{2n}^T\Mb^{-1}J_{2n}^{}( \varUpsilon(t) -X)\big) \varOmega_{ \varUpsilon(t),Y} \\\notag
	&\qquad +  \varOmega_{ \varUpsilon(t),Y}\big(( \varUpsilon(t) - X)^T J_{2n}^T\Mb^{-1}J_{2n}^{}\varUpsilon(t) 
 + X^TJ_{2n}^T\Mb^{-1}J_{2n}^{}( \varUpsilon(t) -X)\big) \\ \notag
    &\qquad + X^T\!J_{2n}^T\Mb^{-1}J_{2n}^{}X( \varOmega_{ \varUpsilon(t),Y} -  \varOmega_{X,Y}) 
    + ( \varOmega_{ \varUpsilon(t),Y} -  \varOmega_{X,Y})X^T\!J_{2n}^T\Mb^{-1}J_{2n}^{}X. \notag
\end{align}
Dividing both sides of this equation by $t$ and taking the limit as $t\to 0$ with the note that $\lim\limits_{t\to 0} \varOmega_{ \varUpsilon(t),Y} =  \varOmega_{X,Y}$ as a consequence of the continuity,  it follows that
\begin{align*}
	2\,\skew( Z^TJ_{2n}^TY) &= 2\,\sym(Z^T\!J_{2n}^T\Mb^{-1}\!J_{2n}^{}X) \varOmega_{X,Y} 
    + 2\,\varOmega_{X,Y}\sym(Z^T\!J_{2n}^T\Mb^{-1}\!J_{2n}^{}X) \\
    & \quad + X^TJ_{2n}^T\Mb^{-1}J_{2n}^{}X\lim\limits_{t\to 0}\frac{1}{t}\left( \varOmega_{ \varUpsilon(t),Y}- \varOmega_{X,Y}\right) \\\notag
    & \quad + \lim\limits_{t\to 0}\frac{1}{t}\left( \varOmega_{ \varUpsilon(t),Y}\!-\! \varOmega_{X,Y}\right) X^TJ_{2n}^T\Mb^{-1}J_{2n}^{}X.
\end{align*}
Then, the limit $\lim\limits_{t\to 0}\frac{1}{t}\left( \varOmega_{ \varUpsilon(t),Y}- \varOmega_{X,Y}\right)$ exists and satisfies a matrix equation which is, by the skew-symmetry of $ \varOmega_{X,Y}$ and symmetry of $\sym(Z^TJ_{2n}^T\Mb^{-1}J_{2n}^{}X)$, is given by~\eqref{eq:eqdirectionalDerivM}.
Thus, the statement follows from \eqref{eq:directionalDeriv_proof1M}. \qed
\end{proof}

Using Lemma~\ref{lem:directionalDerivM}, we obtain the following expression for the Riemannian Hessian of a~function $f$ with respect to the weighted Euclidean metric. 
\begin{theorem}\label{th:hessianM}
Given a $C^2$-function $f:\Spkn\longrightarrow\R$ with its ambient gradient and Hessian $\nabla\bar{f}(X)$ and $\nabla^2\bar{f}(X)$, respectively. The Riemannian Hessian of $f$ at \mbox{$X\in \Spkn$} with respect to the weighted Euclidean metric $g_\Mb$ in~\eqref{eq:metric_euclM} applied to $Z \in \TX$ is computed as
\begin{equation}\label{eq:hessM}
	\HessM{X}{Z} = \Mb^{-1}\big(\Nablaf{X}{Z} -J_{2n} Z \varOmega_{X,\Mb^{-1}\nabla\bar{f}} - J_{2n}X \varTheta_{X,\Mb^{-1}\nabla^2\bar{f}, Z}\big),	
\end{equation}
where $\varOmega_{X,\Mb^{-1}\nabla\bar{f}}$ is the solution to the Lyapunov equation~\eqref{eq:Lyapeq} with $Y=\Mb^{-1}\nabla \bar{f}(X)$ and $\varTheta_{X,\Mb^{-1}\nabla^2\bar{f}, Z}$ is the solution to the Lyapunov equation
\begin{equation*}
    \mathrm{Lyap}_{X,\Mb}(\varTheta) = 2\,\skew\big(X^TJ_{2n}^T\Mb^{-1} \Nablaf{X}{Z} -X^TJ_{2n}^T\Mb^{-1}J_{2n}^{} Z \varOmega_{X,\Mb^{-1}\nabla\bar{f}}\big).
\end{equation*}
\end{theorem}
\begin{proof}
Taking \eqref{eq:Rhessian_M} and \eqref{eq:directionalDerivM} into account, it holds that
\begin{align}
	\HessM{X}{Z}=\ &\projM\big(\Mb^{-1}\Nablaf{X}{Z}  \notag\\
        &
      - \Mb^{-1}J_{2n} ( Z \varOmega_{X,\Mb^{-1}\nabla\bar{f}} 
      + X \varXi_{X,\Mb^{-1}\nabla\bar{f}, Z})\big), \label{eq:hess_proof1M}
\end{align}
where $ \varOmega_{X,\Mb^{-1}\nabla\bar{f}}$ and $ \varXi_{X,\Mb^{-1}\nabla\bar{f}, Z}$ are as in Lemma~\ref{lem:directionalDerivM} with $Y=\Mb^{-1}\nabla\bar{f}(X)$. Since the matrix $ \varXi_{X,\Mb^{-1}\nabla\bar{f}, Z}$ is skew-symmetric, $\Mb^{-1}J_{2n}X \varXi_{X,\Mb^{-1}\nabla\bar{f}, Z}$ is a~normal vector and therefore disappears under the action of the projection $\projM$. This turns the equality \eqref{eq:hess_proof1M} into
\begin{equation}\label{eq:hess_proofM}
	\HessM{X}{Z} = \projM\big(\Mb^{-1}\Nablaf{X}{Z} - \Mb^{-1}J_{2n} Z \varOmega_{X,\Mb^{-1}\nabla\bar{f}}\big).
\end{equation}
Finally, \eqref{eq:hessM} is derived by using the definition of $\projM$ in \eqref{eq:orth_proj_EuclidM}. \qed
\end{proof}

Note that the computation of the Riemannian Hessian \eqref{eq:hessM} with respect to the weighted Euclidean metric $g_\Mb$ involves solving two $2k\times2k$ Lyapunov equations, which are coupled. On the contrary, the counterpart with the canonical-like metric only requires matrix multiplications and inverses. Moreover, setting $\Mb = I_{2n}$, we derive a formula for the Riemannian Hessian with respect to the (classical) Euclidean metric as follows.

\begin{corollary}\label{th:hessian}
Given a $C^2$-function $f:\Spkn\longrightarrow\R$ with its ambient gradient and Hessian $\nabla\bar{f}(X)$ and $\nabla^2 \bar{f}(X)$, respectively. The Riemannian Hessian of $f$ at $X\in \Spkn$ with respect to the Euclidean metric 
applied to $ Z \in \TX$ is computed as
\begin{align*}
	\Hesse{X}{Z} &= \proje\big(\Nablaf{X}{Z} - J_{2n} Z \varOmega_{X,\nabla\bar{f}}\big)\\ 
 &= \Nablaf{X}{Z} -J_{2n} Z \varOmega_{X,\nabla\bar{f}} - J_{2n}X \varTheta_{X,\nabla^2\bar{f}, Z}, 
 \end{align*}
where $ \varOmega_{X,\nabla\bar{f}}$ and $\varTheta_{X,\nabla^2\bar{f}, Z}$ are the solutions to the Lyapunov equations 
\begin{align*} 
	X^TX \varOmega +  \varOmega X^TX 
 & = 2\,\skew\big(X^TJ_{2n}^T\nabla\bar{f}(X)\big),\qquad\qquad\qquad\\
	X^TX \varTheta +  \varTheta X^TX 
 & = 2\,\skew\big(X^TJ_{2n}^T \Nablaf{X}{Z} -X^T Z \varOmega_{X,\nabla\bar{f}}\big),
\end{align*}
respectively.
\end{corollary}

\begin{remark}\label{rem:Hess_M}
	The Riemannian Hessian $\hess_\Mb f(X)$ in \eqref{eq:hessM} can also be viewed as a~bilinear form on $\TX\times \TX$. Given \mbox{$Z,U\in \TX$}, as the last term in \eqref{eq:hessM} belongs to the normal space at $X$ given in~\eqref{eq:normalspaceM}, we get 
	\begin{align*}
		\hess_\Mb f(X)[Z, U] &= \langle U,\HessM{X}{Z}\rangle \\
  & = \tr\big(U^T\Mb^{-1}(\Nablaf{X}{Z} -J_{2n} Z \varOmega_{X,\Mb^{-1}\nabla\bar{f}})\big).
	\end{align*}
	In the particular case $k=n$ and $\Mb=I_{2n}$, this expression coincides with that  presented in \textup{\cite[Thm.~3]{BirtCC20}}. Thus, our findings extend the results of \textup{\cite{BirtCC20}}, derived for the symplectic group $\Spn$, to the symplectic Stiefel manifold $\Spkn$. 
\end{remark}

%----------------------------------------------------------------------%
%
\section{Riemannian Newton methods}
\label{sec:Solve_NewtonEq}
In this section, we present the Riemannian Newton method for solving the constrained minimization problem \eqref{eq:opt_prob} and discuss its variants; the convergence properties will be studied in the next section. To enable the presentation, we need a~retraction which allows moving in the direction of a~tangent vector while staying on the manifold. Note, however, that the proposed Newton's methods in this section as well as their convergence property do not depend on a~specific retraction. 

Arising as an approximation to the exponential mapping, which is not always computationally tractable, retractions are becoming a~reliable tool in Riemannian optimization. Denote by $\romT\Spkn$ the~tangent bundle of $\Spkn$. A~smooth mapping \mbox{$\mathcal{R}:\romT\Spkn \to \Spkn$} is called a~\emph{retraction} if for all \mbox{$X\in\Spkn$}, the restriction of $\mathcal{R}$ to $\TX$, denoted by $\mathcal{R}_X$, 
satisfies the following properties:
\begin{enumerate}
	\item[1)] $\mathcal{R}_X(0_X)=X$, where $0_X$ denotes the origin of $\TX$;
	\item[2)] $\left.\tfrac{{\rm d}}{{\rm d}t} \mathcal{R}_X(t\,Z)\right|_{t=0}=Z$ for all $Z\in \TX$.
\end{enumerate}

Several retractions on the symplectic Stiefel manifold $\Spkn$ have been
constructed in \cite{GSAS21,GSS24,BendZ21,OviH23,JenZ24}. Here, we briefly review only those which will be used in the numerical experiments in Section~\ref{sec:Numer}. 

\paragraph{Cayley retraction}
Based on the Cayley transformation, the \emph{Cayley retraction} is given by
\[
\calR_X^{\mathrm{cay}}(Z) = \Big(I_{2n}-\frac{1}{2}S_{X,Z}J_{2n}\Big)^{-1}
\Big(I_{2n} + \frac{1}{2}S_{X,Z}J_{2n}\Big)X
\]
with $S_{X,Z} = G_XZ(XJ_{2k})^T + XJ_{2k}(G_XZ)^T$ and $G_X = I_{2n}-\frac{1}{2}XJ_{2k}^{}X^TJ_{2n}^T$. It has been shown in \cite[Prop.~5.4]{GSAS21} that $\calR_X^{\mathrm{cay}}(tZ)$ is defined for any \mbox{$t\geq 0$} if and only if $S_{X,Z}J_{2n}$ has no nonzero real eigenvalues, i.e., it is locally defined. If $k$ is considerably smaller than $n$, then in addition to \cite[(5.7)]{GSAS21}, the authors of \cite{BendZ21,OviH23} suggested an~economical formula  
\begin{equation}\label{eq:CayRetr_econ}
\mathcal{R}_X^{\mathrm{cay}}(Z) = 
-X+\bigl(P_X Z+2X\bigr)\Bigl(I_{2k}+\frac{1}{4}J_{2k}^TZ^TJ_{2n}^{}(P_X Z+2X)\Bigr)^{-1}
\end{equation}
with $P_X$ as in \eqref{eq:prPXperp}.
The computation of this retraction requires solving a~linear system with a~\mbox{$2k\times 2k$} matrix only. Note that the Cayley retraction belongs to a~family of retractions recently introduced in~\cite{OviH23}. Another Cayley transformation-based retraction has been proposed in \cite{JenZ24}, which provides a~better approximation to the geodesic but it is more expensive to compute compared to the retraction in~\eqref{eq:CayRetr_econ}.

\paragraph{SR retraction}
Recalling that a~matrix $A\in\Rbnk$ having full rank can in general be decomposed as $A = SR$, where $S\in\Spkn$ and $R\in\R^{2k\times 2k}$ is congruent to an~upper triangular matrix by a~permutation matrix 
\[
P_{2k}^{ } = [\e_1, \e_3,\ldots, \e_{2k-1},\e_2,\ldots,\e_{2k}] \in \mathbb{R}^{2k\times 2k},
\]
where $\e_j\in\mathbb{R}^{2k}$ denotes the $j$-th unit vector, see \cite[Sect.~3.3]{GSS24} for detail. If $R$ is additionally restricted to the matrix set
\begin{equation*}
\arraycolsep=2pt
\begin{array}{rl}
	T_{2k}^0(P_{2k}) =\bigl\{P_{2k}^T\hat{R}P^{ }_{2k}\ :\ 
	\hat{R}=[r_{ij}]\!\in\! &\!\mathbb{R}^{2k\times 2k} \text{ is upper triangular with } r_{2j-1,2j} = 0\bigr. \\ 
	&  \bigl.  
	 \text{ and } 
	|r_{2j,2j}|=r_{2j-1,2j-1}>0, \,   
	j = 1,\ldots, k\bigr\},
\end{array}
\end{equation*}
this decomposition is unique. Such a~decomposition can be computed by using a~symplectic Gram--Schmidt procedure \cite{Sala05}. Based on the SR decomposition, the  \emph{SR retraction} is then defined as
\begin{equation}\label{eq:SRRetr}
\calR_X^{\mathrm{SR}}(Z) = \mbox{sf}(X+Z), 
\end{equation}
where $\mbox{sf}(\cdot)$ denotes the symplectic factor $S$ in the SR decomposition. It follows from \cite[Thm.~3.3]{GSS24} that if $\|Z\|_{\Mb_X} < \sqrt{\lambda_{\min}(\Mb_X)}$,
where $\lambda_{\min}(\Mb_X)$ denotes the smallest eigenvalue of $\Mb_X$,  
%{\cite[Thm.~3.3]{GSS24} requires $\|Z\|_2<1$. We have $\|Z\|_2\leq \|Z\|_{\mathrm{F}}= \|\Mb_X^{-1/2}\Mb_X^{1/2}Z\|_{\mathrm{F}}\leq \|\Mb_X^{-1/2}\|_2\|\Mb_X^{1/2}Z\|_{\mathrm{F}}=\frac{1}{\sqrt{\lambda_{\min}(\Mb_X)}}\|Z\|_X<1$. The condition $\|Z\|_X < \sqrt{\lambda_{\min}(\Mb_X)}$ is stronger than $\|Z\|_2<1$, but it is related to the metric} 
then $X+Z$ has an~SR decomposition and \eqref{eq:SRRetr} is uniquely determined. This means that the SR retraction is locally defined. 

\medskip
We are now ready to state the Riemannian Newton method. In a~general setting, starting with an~initial guess $X_0\in\Spkn$,  the Newton search direction is computed by solving the Newton equation
\begin{equation}\label{eq:Newtoneq_general}
	\Hess{X_j}{Z_j} = -\grad f(X_j)
\end{equation}
for $Z_j \in \TX$. Then the iterate is updated as $X_{j+1}=\mathcal{R}_{X_j}(Z_j)$ with a~retraction $\mathcal{R}$ defined on $\Spkn$. The resulting Riemannian Newton method is summarized in Algorithm~\ref{alg:Newton_plain}.

\begin{algorithm}[htbp]
\caption{Riemannian Newton method (RN)}
\label{alg:Newton_plain}
\begin{algorithmic}[1]
	\Require Starting point $X_0\in\Spkn$, maximal number of iterations \texttt{mxit} and the stopping criterion \texttt{tol}. Set $j=0$.
	\While{$\|\grad f(X_j)\|_{\Mb_{X_j}} >$ \texttt{tol} and $j <$ \texttt{mxit}}
	\State  Solve the Newton equation $\Hess{X_j}{Z_j}\! = -\grad f(X_j)$ for $Z_j\!\in\! \TX$.\!\!
	\State Update $X_{j+1}  = \calR_{X_j}( Z_j)$.
	\State Set $j=j+1$. 
	\EndWhile
\end{algorithmic}
\end{algorithm}

The heart of the Riemannian Newton method is to solve the Newton equation~\eqref{eq:Newtoneq_general}. Unfortunately, its left-hand side is rather complex: either by a long formula or being extremely implicit via coupled Lyapunov equations. In the subsequent subsections, we propose different frameworks to address this problem.

%-----------------------------------------------------------------------------%
%
\subsection{Saddle point problem}
\label{ssec:saddle_point}
We now study the Newton equation \eqref{eq:Newtoneq_general} in more detail. For the sake of brevity, we consider the weighted Euclidean metric only. The canonical-like metric \eqref{eq:canonical_metric} can be treated analogously. Furthermore, to shorten the notation, we omit the subscript~$j$. 

Using \eqref{eq:hess_proofM}, the Newton equation \eqref{eq:Newtoneq_general} can be written as
\begin{equation}\label{eq:Newtoneq}
\projM\big(\Mb^{-1}\Nablaf{X}{Z} -\Mb^{-1}J_{2n} Z \varOmega_{X,\Mb^{-1}\nabla\bar{f}}\big) 
    = -\rgradM{X}
\end{equation}
with unknown $Z \in \TX$.
We aim to turn this equation into a~saddle point problem by getting rid of the projection $\projM$. To this end, let us introduce a~linear operator 
\begin{equation}\label{eq:LinOp}
\begin{array}{rccl}
	\varPsi_{f,X}\ :&\Rbnk &\longrightarrow &\Rbnk\\
	& Y &\longmapsto & \Mb^{-1}\Nablaf{X}{Y} -\Mb^{-1}J_{2n} Y \varOmega_{X,\Mb^{-1}\nabla\bar{f}}.
\end{array}
\end{equation}
In view of $Z \in \TX$, it follows that $Z=\projM (Z)$. Combining with the fact that $\rgradM{X}=\projM\big(\Mb^{-1}\nabla \bar{f}(X)\big)$, equation~\eqref{eq:Newtoneq} is equivalent to 
\begin{equation}\label{eq:Newtoneq_2}
\projM \varPsi_{f,X}^{} \projM (Z) = \projM(G)
\end{equation}
with $G=-\nabla\bar{f}(X)$.
The following proposition shows that the solution to this equation can be determined by solving a~saddle point problem. 

\begin{proposition}\label{prop:NewtonEq-SaddlePoint}
Let $X\in \Spkn$, $ G=-\nabla\bar{f}(X)$, and let $ \varPsi_{f,X}$ be defined as in~\eqref{eq:LinOp}. 
If $Z\in\TX$ is a~solution to the Newton equation \eqref{eq:Newtoneq_2},
then $(Z, \varOmega)$ with 
\begin{equation}\label{eq:solU}
	\varOmega = \big(\D F_X^{} \Mb^{-1}\D F_X^*\big)^{-1}\D F_X^{} \big(G - \varPsi_{f,X}(Z)\big) \in \skewset(2k)
\end{equation}
is a~solution to the saddle point problem
\begin{subequations}\label{eq:saddle_point}
	\begin{alignat}{3}\label{eq:saddle_point_1}
		& \varPsi_{f,X}(Z) + \Mb^{-1}\D F_X^*(\varOmega) &\,=G,\\ \label{eq:saddle_point_2}
		&\D F_X^{}(Z) &=0.
	\end{alignat}
\end{subequations}
Conversely, if $(Z, \varOmega) \in \Rbnk \times \skewset(2k)$ is a~solution to the saddle point prob\-lem~\eqref{eq:saddle_point}, then $ Z\in \TX$ and it solves the Newton equation \eqref{eq:Newtoneq_2}.
\end{proposition}
\begin{proof}
Let $Z\in\TX$ be a~solution to the Newton equation \eqref{eq:Newtoneq_2} and let~$\varOmega$ be as in \eqref{eq:solU}. Then using \eqref{eq:Du_orth_proj} with $\Mb_{X}=\Mb$, we obtain that 
\begin{align*}
	\varPsi_{f,X}(Z) + \Mb^{-1}\D F_X^{*}(\varOmega) 
 %&= \varPsi_{f,X}(Z) + \Mb^{-1}\D F_X^*(\D F_X^{}\ts{\Mb^{-1}}\D F_X^*)^{-1}\D F_X^{}\big(G - \varPsi_{f,X}(Z)\big)\\
	&= \big(\projM \varPsi_{f,X} \projM (Z) - \projM (G)\big) +  G =  G.
\end{align*}
Moreover, equation \eqref{eq:saddle_point_2} immediately follows from \eqref{eq:tangentspace_1}.

For the converse statement, employing again \eqref{eq:Du_orth_proj} and \eqref{eq:saddle_point_2}, we observe that  
\[
\projM (Z) =  Z - \Mb^{-1}\D F_X^*\big(\D F_X^{} \Mb^{-1}\D F_X^*\big)^{-1}\D F_X^{}(Z) =  Z
\]
which implies that $Z$ belongs to $\TX$. Further, rewriting equation \eqref{eq:saddle_point_1} as 
$\varPsi_{f,X}(\projM(Z)) + \Mb^{-1}\D F_X^*(\varOmega) =  G$ and letting $\projM$ act on both sides of this equation, we obtain by using $\projM \big(\Mb^{-1}DF_X^*(\varOmega)\big)=0$ that the Newton equation~\eqref{eq:Newtoneq_2} holds true. \qed
\end{proof}

Due to the complexity of the saddle point equation \eqref{eq:saddle_point}, 
an~explicit direct matrix solver is unavailable. For small to medium-sized problems, we therefore propose to vectorize the involved equations and to solve the resulting linear system. This remedy was also used for the Riemannian Newton method on the Stiefel manifold in \cite{Sato17}. For this purpose, we introduce the vectorization operators on the matrix spaces $\mathbb{R}^{2n\times 2k}$ and $\skewset(2k)$. Let $\mvec{Z}\in\mathbb{R}^{4nk}$ denote the column vector generated by vertically concatenating the columns of the matrix $Z\in\mathbb{R}^{2n\times 2k}$. Further, let \mbox{$\mveck{\varOmega}\in \mathbb{R}^{k(2k-1)}$} denote the column vector constructed by vertically concatenating the columns of the upper triangular part (excluding the diagonal) of the matrix $\varOmega\in \skewset(2k)$. The following proposition collects some useful properties of these vectorization operators and the Kronecker product. Their proofs  can be found in \cite[Chap.~4]{HornJ91}. 
\begin{proposition}\label{prop:vect}
	Let $Z\in \mathbb{R}^{2n\times 2k}$, $A\in \mathbb{R}^{2k\times 2n}$, $B\in\mathbb{R}^{2k\times 2k}$, $\varOmega\in \skewset(2k)$, and let the matrix $E_{ij}\in \R^{2n\times 2k}$ have entry $1$ in position $(i,j)$ and $0$ otherwise.
 Furthermore, let $P_{2n,2k} = \sum_{i=1}^{2n}\sum_{j=1}^{2k}E_{ij}\otimes E^T_{ij}\in\R^{4nk\times 4nk}$ be a permutation matrix, where $\otimes$ denotes the Kronecker product. Then we have
	\begin{enumerate}
		\item $\mvec{Z^T} = P_{2n,2k}\mvec{Z}$;
		\item $\mvec{AZB} = (B^T\otimes A)\mvec{Z}$;
		\item $(B^T\otimes A)P_{2n,2k}= P_{2k,2k}(A\otimes B^T)$; 
		\item $\mvec{\varOmega}=D_{2k}\mveck{\varOmega}$ and $\mveck{\varOmega}=\frac{1}{2}D_{2k}^T\mvec{\varOmega}$ with the duplication matrix $D_{2k}\in\mathbb{R}^{4k^2\times k(2k-1)}$ satisfying $D_{2k}^T=-D_{2k}^TP_{2k,2k}^{}$.
	\end{enumerate}
\end{proposition}

Applying the corresponding vectorization operators to equations \eqref{eq:saddle_point} and exploiting the skew-symmetry of the matrices $\varOmega$ and $\D F_{X}(Z)$, we can reformulate the saddle point problem~\eqref{eq:saddle_point} by using Proposition~\ref{prop:vect} as the linear system in $\mathbb{R}^{4nk+k(2k-1)}$ given by 
\begin{equation}\label{eq:LinSys}
	\begin{bmatrix}
		A &\enskip B\\
		C &\enskip 0
	\end{bmatrix}\begin{bmatrix}
		z\\
		\omega
	\end{bmatrix} = \begin{bmatrix}
		g\\
		0
	\end{bmatrix},
\end{equation}
where $z=\mvec{Z}$, $\omega=\mveck{\varOmega}$, $g=\mvec{G}$, and
\begin{align*}
	A &= I_{2k}\otimes \Mb^{-1}\nabla^2\bar{f}(X)-  \varOmega_{X,\Mb^{-1}\nabla\bar{f}}^T\otimes \Mb^{-1}J_{2n},\\
	B & = -2\big(I_{2k}\otimes(\Mb^{-1}J_{2n}X)\big)D_{2k},\\
	C & = \frac{1}{2}D_{2k}^T\big(P_{2k,2k}-I_{4k^2}\big)\big(I_{2k}\otimes(J_{2n}X)^T\big) 
 = -D_{2k}^T\big(I_{2k}\otimes(J_{2n}X)^T\big).
\end{align*}
If the matrices $A$ and $CA^{-1}B$ are both nonsingular, then 
\eqref{eq:LinSys} is uniquely solvable and the solution $z$ of system~\eqref{eq:LinSys} has the expression
\[
z = A^{-1}g - A^{-1}B(CA^{-1}B)^{-1}CA^{-1}g,
\]
see \cite[Sect.~3.3]{BenzGL05}. For $k \ll n$, the most expensive part in computing this solution is solving $k(2k-1)+1$ linear systems with the matrix $A$ of size $4nk\times 4nk$.

Although the presented approach is applicable to small and medium-sized problems only, it provides a~closed-form solution, which can be adopted as a~reference solution for inexact methods which will be discussed next.

\begin{remark}
For a dense weighted matrix $\Mb$, the inverse in \eqref{eq:saddle_point} might be computationally expensive. One way to circumvent this is as follows. Multiplying \eqref{eq:saddle_point_1} with $M$ and \eqref{eq:saddle_point_2} with $2$ yields
\begin{subequations}
	\begin{alignat*}{3}
		& \Mb\varPsi_{f,X}(Z) + \D F_X^*(\varOmega) &\,= \Mb G,\\ \label{eq:saddle_point_2M}
		& 2\,\D F_X^{}(Z) &=0.\quad\enskip
	\end{alignat*}
\end{subequations}
Similarly, for the vectorized system~\eqref{eq:LinSys}, if we write $g=\mvec{\Mb G}$,
\begin{align*}
	A &= I_{2k}\otimes \nabla^2\bar{f}(X)-  \varOmega_{X,\Mb^{-1}\nabla\bar{f}}^T\otimes J_{2n},\\
	B & = -2\big(I_{2k}\otimes(J_{2n}X)\big)D_{2k},\\
	C & = D_{2k}^T\big(P_{2k,2k}-I_{4k^2}\big)\big(I_{2k}\otimes(J_{2n}X)^T\big) 
        = -2 D_{2k}^T\big(I_{2k}\otimes(J_{2n}X)^T\big) = B^T,
\end{align*}
then we get again $z = - A^{-1}B(B^TA^{-1}B)^{-1}B^TA^{-1}g + A^{-1}g$. 
\end{remark}

%--------------------------------------------------------------------------------%
%
\subsection{Riemannian inexact Newton method}
An alternative and more commonly used approach for solving the Newton equation~\eqref{eq:Newtoneq_general} is to employ an~iterative method such as the conjugate gradient (CG) or minimal residual (MINRES) method, e.g., \cite{Saad96}. Note that CG requires the property of symmetry and positive definiteness while MINRES only needs the symmetry of the coefficient matrix or operator. 

In these methods, starting with an~initial guess $Z_j^{(0)}\in \TX$, a sequence $\{Z_j^{(i)}\}_i$ of approximate solutions to \eqref{eq:Newtoneq_general} is generated, which requires only the act of the Riemannian Hessian on a tangent vector and linear operations, and terminates once the stopping criterion 
\[	
	\big\|\Hess{X_j}{Z_j^{(i)}} + \grad f(X_j) \big\|_{\Mb_{X_j}} 
        \leq  \eta_j\,\big\|\grad f(X_j)\big\|_{\Mb_{X_j}}
\] 
is fulfilled for some $\eta_j\in (0,1)$. For the convergence reasons, this so-called forcing term is chosen as 
\[
    \eta_j = \min\big\{\eta, \|\grad f(X_j)\|^\mu_{\Mb_{X_j}}\!\big\}
\]
with some fixed constants \mbox{$\eta \in (0,1)$} and $\mu>0$. We summarize the inexact version of the Riemannian Newton method in Algorithm~\ref{alg:Newton_inexact}. 
		
\begin{algorithm}[ht]
		\caption{Riemannian inexact Newton method (RiN)}
		\label{alg:Newton_inexact}
	\begin{algorithmic}[1]
			\Require Starting point $X_0\in\Spkn$, 
			parameters for inexact solver $\eta\in (0,1)$, $ \mu >0$, maximal number of iterations \texttt{mxit} and the stopping criterion \texttt{tol}. Set $j=0$.
			\While{$\|\grad f(X_j)\|_{\Mb_{X_j}} >$ \texttt{tol} and $j <$ \texttt{mxit}}
			\State  Solve the Newton equation \eqref{eq:Newtoneq_general} approximately for $\tilde{Z}_j\in\TX$ until the 
   
   following condition is satisfied:
   \[
			\quad\big\|\Hess{X_j}{\tilde{Z}_j} + \grad f(X_j) \big\|_{\Mb_{X_j}} \!\!\leq \min\big\{\eta,\big\|\grad f(X_j)\big\|_{\Mb_{X_j}}^\mu\!\!\big\}\big\|\grad f(X_j)\big\|_{\Mb_{X_j}}\!. \vspace*{-2mm}
    \]
			\State Update $X_{j+1} = \calR_{X_j}(\tilde{Z}_j)$. 
   		\State Set $j=j+1$.
			\EndWhile
	\end{algorithmic}
\end{algorithm}
	
In the weighted Euclidean metric case, one can alternatively apply an~iterative method to the saddle point problem~\eqref{eq:saddle_point} which can be considered as a~linear system on \mbox{$\R^{2n\times 2k}\times \skewset(2k)$}. 
This problem, however, has a~larger dimension than the Newton equation \eqref{eq:Newtoneq}. Moreover, in our numerical experiments, we observed that the numerical solution of \eqref{eq:saddle_point} poses some stability issues leading to less accurate results compared to that obtained by solving the Newton equation  \eqref{eq:Newtoneq}. Therefore, the formulation \eqref{eq:Newtoneq} (or more general \eqref{eq:Newtoneq_general}) will be used whenever an~iterative solver for the Newton equation is invoked.

 %----------------------------------------------------------------------
\subsection{Hybrid Riemannian Newton method} 
\label{sec:RNewton_meth}
The Riemannian Newton method, as presented in Algorithms~\ref{alg:Newton_plain} and \ref{alg:Newton_inexact}, does not ensure a convergence without a requirement that the initial guess is close enough to a~solution. 

The so-called \emph{globalized approach} as presented in \cite{ZhaoBJ15,ZhaoBJ18,BortFFY20,BortFF22,XuNgB22} remedies this fact by replacing any unsatisfactory Newton direction by a~gradient descent one and, in addition, by controlling the step length to enable a~decrease in the value of the cost function. That is, at every iterate, the Riemannian Hessian must be computed and Newton's equation must be solved. Both theory and practice show that if the initial guess is not close enough to a solution, gradient descent is usually needed. Accordingly, the computation of the Riemannian Hessian and the solution to Newton's equation are usually unnecessary in the early phase of the iteration which most probably slows down the optimization process. Moreover, in the late phase when the iterate is close to a solution, the step size control which requires retracting the search direction and calculating the cost function is unjustified due to the fact that the Newton step length in this phase tends to be unit, see, e.g., \cite[Lem.~4.5]{ZhaoBJ15}, \cite[Lem.~5]{ZhaoBJ18}, \cite[proof of Thm.~4.2]{BortFF22}.

In view of these, we employ a \emph{hybrid method}  \cite{SatoI13,IzmaS14} which is composed of two phases. In the first phase, the search is started with gradient-based steps until a~switching condition is fulfilled, e.g., the norm of the Riemannian gradient is smaller than a~prescribed constant. For this purpose, we can resort to the RGD method proposed in~\cite[Alg.~5.1]{GSAS21} but with slight adaptation to the metric $g_{\Mb_{X}}$.
The goal of the first phase is to bring the iterate into a small enough neighborhood of a~critical point. Once it is accomplished, the second phase is activated, in which Newton's steps are performed as presented in Algorithm~\ref{alg:Newton_plain} or Algorithm~\ref{alg:Newton_inexact} to speed up the local convergence. Details of the hybrid Riemannian Newton method are given in Algorithm~\ref{alg:Newton_hybrid}. 
\begin{algorithm}[htbp]
		\caption{Hybrid Riemannian Newton method (hRN)}
		\begin{algorithmic}[1]
			\Require Starting point $X_0\in\Spkn$, switching parameter $\theta \in (0,1)$, maximal number of iterations \texttt{mxit} and the stopping criterion $\texttt{tol}<\theta$.
                \State Run the RGD method with respect to $g_{\Mb_{X}}\!$ 
                 and obtain ${X}_{j_*}$.
                 \Comment{First phase: RGD}	
                \If{$\|\grad f({X}_{j_*})\|_{\Mb_{{X}_{j_*}}} \leq \theta$}
                \State \hspace*{-3mm}Run Algorithm~\ref{alg:Newton_plain} or Algorithm~\ref{alg:Newton_inexact} starting with ${X}_{j_*}$. \Comment{Second~phase:~Newton}
			\Else
			\State  \hspace*{-3mm}Release error, increase $\theta$ (or increase \texttt{mxit} for RGD).
			\EndIf
		\end{algorithmic}
		\label{alg:Newton_hybrid}
\end{algorithm}
 
%------------------------------------------------------------------------
\section{Convergence analysis}
 \label{sec:Convergence}
 This section is devoted to the convergence analysis of the proposed algorithms. First, we study the local convergence properties of Algorithms~\ref{alg:Newton_plain} and \ref{alg:Newton_inexact} and then, plugging them in Algorithm~\ref{alg:Newton_hybrid}, we consider the global behavior. 
 
\subsection{Local convergence}
The local quadratic convergence rate of Algorithm~\ref{alg:Newton_plain} can be derived from \cite[Thm.~6.3.2]{AbsiMS08} or \cite[Thm.~6.7]{Boumal23} as it is a~special case of the Riemannian Newton method for general Riemannian manifolds. The convergence analysis of the Riemannian inexact Newton method is, however, more involved. To this end, we first collect some facts related to exponential mapping from \cite{doCarmo92,FerrS02,FernFY17,ZhuS20}. By extending them to general retractions, we establish the superlinear convergence of the inexact Newton method. For the sake of convenience, we state all results on the symplectic Stiefel manifold, although they hold for a~general Riemannian ma\-ni\-fold. In addition, we simplify the notation for a~general norm $\|\cdot\|_{\Mb_X}$ by $\|\cdot\|_{X}$ if there is no ambiguity.

\paragraph{Exponential mapping}
The {\em exponential mapping}  $\exp_X$ is defined as
\[
\arraycolsep=2pt
\begin{array}{rccc}
\exp_X\;: \;\;&\TX&\longrightarrow &\Spkn\\ 
&Z&\longmapsto&  \gamma(1),
\end{array}
\]
where $\gamma(t)$ is the (unique) geodesic which satisfies the conditions $\gamma(0) =X$ and $\dot{\gamma}(0) = Z\in \TX$. In general, the exponential map is defined only on a~neighborhood of the origin \mbox{$0_X\in \TX$}, e.g., \cite[Prop.~2.7 of Chap.~3]{doCarmo92}. Let us introduce the (\emph{exponential}) \emph{injectivity radius} given by
\begin{equation*}
r^{\exp}_X = \sup\big\{r > 0\ :\ \exp_X\big|_{B_r(0_X)} \mbox{ is diffeomorphic}\big\}.
\end{equation*}
where $B_r(0_X):=\{Z\in\TX : \|Z\|_X<r\}$. 
For any $r<r^{\exp}_X$, the set $\exp_X(B_r(0_X))\subset \Spkn$ is then referred to as a~\emph{normal ball} or the \emph{geodesic ball} due to the fact that $\exp_X(B_r(0_X)) = \{Y\in\Spkn : \dist(X,Y)<r\}$, where $\dist(X,Y)$ denotes the Riemannian distance between $X$ and $Y$ on $\Spkn$. A~normal ball is a~specific case of a~\emph{normal neighborhood} (NN), 
the diffeomorphic image of a~star-shaped neighborhood of $0_X$ in $\TX$ under $\exp_X$. In a~NN of $X$, $\exp_X^{-1}(Y)$ is well-defined but $\exp_Y^{-1}(X)$ might not. A~\emph{totally normal neighborhood} (TNN) of $X$ excludes such cases by ensuring that it is the NN of all of its points, see \cite[Thm.~3.7 of Chap.~3]{doCarmo92} for the existence of such a~neighborhood. This theorem also implies that given any two points $X$ and $Y$ in a~TNN of some $\tilde{X}\in\Spkn$, the inverse exponential map is well-defined and the curve $\gamma_{XY}(t):=\exp_X\big(t\exp^{-1}_X(Y)\big)$, $t\in [0,1]$, is the unique geodesic connecting $X$ and~$Y$. Such a~geodesic allows us to define a~parallel transport from $X$ to $Y$ in an~appropriate way, see \cite[Chap.~2]{doCarmo92} for a~precise definition. The parallel transport along the curve $\gamma_{XY}$, denoted by ${\calT}^{\,\|}_{XY}$, is uniquely defined and isometric. 
Moreover, it follows from \cite[(18)]{ZhuS20} that 
\begin{equation}\label{eq:TransInvExp}
    -\calT^{\,\|}_{YX}\exp^{-1}_{Y}(X) = \exp^{-1}_{X}(Y).
\end{equation}
From now on, arguments related to parallel transport $\calT^{\,\|}_{XY}$ always go with the assumption that $X$ and $Y$ are in a TNN of some $\tilde{X}\in\Spkn$ which immediately implies that $\dist(X,Y)<\min\big\{r_X^{\exp},r_Y^{\exp}\big\}$.

Parallel transports help to relate geometric objects associated with different tangent spaces. Indeed, if $f$ is twice continuously differentiable on $\Spkn$, then it has been shown in \cite[Lem.~3.2]{FernFY17} that 
\begin{equation}\label{eq:PhessP}
			\lim_{X\to \tilde{X}}\big\|{\calT}^{\,\|}_{X\tilde{X}} \hess f(X){\calT}^{\,\|}_{\tilde{X}X} - \hess f(\tilde{X})\big\|_{\tilde{X}} = 0.
\end{equation}
If $\hess f(\tilde{X})$ is additionally nonsingular, then there exists a~constant $\tilde{r} \in (0,r^{\exp}_{\tilde{X}})$ such that for all $X\in\Spkn$ with $\dist(X,\tilde{X})< \tilde{r}$, $\hess f(X)$ is nonsingular and 
\begin{equation}\label{eq:hessnonsing}
		\big\|\hess f(X)^{-1}\big\|_{X} \leq 2\big\|\hess f(\tilde{X})^{-1}\big\|_{\tilde{X}}.
\end{equation}
Moreover, \cite[Lem.~2.3]{FerrS02} implies that the first-order approximation formula 
\begin{equation}\label{eq:Taylor_VectorField}
		\grad f(X) = {\calT}^{\,\|}_{\tilde{X}X}\grad f(\tilde{X}) 
        + {\calT}^{\,\|}_{\tilde{X}X}\Hess{\tilde{X}}{\exp^{-1}_{\tilde{X}}(X)} 
        + o\big(\|\exp^{-1}_{\tilde{X}}(X)\|_{\tilde{X}}\big)
\end{equation}
holds for any $X\in\Spkn$ close enough to $\tilde{X}$.

\paragraph{Retraction-related results}
We consider now a~retraction $\calR$ on $\Spkn$ which approximates the exponential mapping to the first order.  By definition, any retraction is a~local diffeomorphism. The positive number
\[
	r_X := \sup\big\{r > 0\ :\ \calR_X\big|_{B_r(0_X)} \text{ is diffeomorphic}\big\}
\]
is called the~\emph{injectivity radius} of the retraction~$\mathcal{R}$ at $X \in \Spkn$. It follows from \cite[Lem.~2]{HuangAG2015} that for any \mbox{$\tilde{X}\in \Spkn$}, there exist positive constants $c_1$, $c_2$, $\delta^{\tan}_{\tilde{X}}$, and $\delta_{\tilde{X}}$ such that for all $X\!\in \Spkn$ satisfying  
\mbox{$\dist(X,\tilde{X})\!<\delta_{\tilde{X}}$} and all $Z_j\in \TX$ with $\|Z_j\|_X\leq \delta^{\tan}_{\tilde{X}}$, $j=1,2$, it holds that 
\begin{equation}\label{eq:DistRetr_Estimate}
c_1\|Z_1-Z_2\|_{X} \leq \dist(\calR_X(Z_1),\calR_X(Z_2)) \leq c_2\|Z_1-Z_2\|_{X}.
\end{equation}
Setting $Z_1=Z$ and $Z_2 = 0_{X}$ in the above estimates, we immediately obtain
\begin{equation}\label{eq:DistRetr_Estimate_1}
			c_1\|Z\|_{X} \leq \dist(\calR_X(Z),X) \leq c_2\|Z\|_{X}.
\end{equation}
Furthermore, without loss of generality, one can assume that $\delta^{\tan}_{\tilde{X}}$ in \eqref{eq:DistRetr_Estimate_1} is smaller than the injectivity radius $r_X$, as \eqref{eq:DistRetr_Estimate_1} holds for any other constant smaller than $\delta^{\tan}_{\tilde{X}}$. Then \eqref{eq:DistRetr_Estimate_1} yields
\begin{equation}\label{eq:DistRetr_Estimate_2}
c_1\big\|\calR^{-1}_X(Y)\big\|_{X} \leq \dist(Y,X) \leq c_2\big\|\calR^{-1}_X(Y)\big\|_{X}
\end{equation}
for any $Y\in \Spkn$ such that $\dist(Y,X)<\delta$ with $\delta < \delta_{\tilde{X}}$. 
The meaning of~\eqref{eq:DistRetr_Estimate_2} is that if $X$ and $Y$ are close enough to each other, the Riemannian distance between them and the norm of pre-image of $Y$ under the retraction at $X$ are of the same order which turns out to be helpful later. Estimate \eqref{eq:DistRetr_Estimate} allows us also to define a~quantity \cite{DediPM03,FernFY17,BortFF22} 
\begin{equation}\label{eq:Retr_LipschitzConst}
K_{\calR,\tilde{X}}:= \sup\left\{\frac{\dist(\calR_X(Z_1),\calR_X(Z_2))}{\|Z_1-Z_2\|_{X}}\ :\ \dist(X,\tilde{X})<\delta_{\tilde{X}},\ Z_1\!\not=\! Z_2\in\! B_{\delta^{tan}_{\tilde{X}}}(0_X) \right\},
\end{equation}
where $\delta_{\tilde{X}}$ and $\delta^{\tan}_{\tilde{X}}$ are given in \eqref{eq:DistRetr_Estimate} which in turn implies that $K_{\calR,\tilde{X}}$ is positive and finite. 

Similarly to the exponential mapping, for a given retraction $\calR$, we define at each point the \emph{retractive neighborhood} (RtrN) and the \emph{totally retractive neighborhood} (TRtrN). For any two points $X$ and $Y$ in a~TRtrN of some
$\tilde{X}\in\Spkn$, the inverse retraction $\calR_X^{-1}(Y) \in \TX$ is well-defined. By definition, the intersection of an~RtrN (or a~TRtrN)  at $X$ with its normal counterpart is always nonempty and open since it is the nonempty intersection (containing $X$ at least) of two open subsets of $\Spkn$. Because the retraction is designed to approximate the exponential map, it would be curious to see how the inverse retraction approximates the inverse exponential. As shown in \cite[Lem.~3]{ZhuS20}, there exists a~constant $c>0$, depending on $\tilde{X}\in\Spkn$, such that for all $X$ and~$Y$ belonging to a~compact subset of the intersection of the TNN and TRtrN of~$\tilde{X}$, it holds that
\begin{equation}\label{eq:ApproxInvRetr}
\big\|\calR_X^{-1}(Y) - \exp_X^{-1}(Y)\big\|_X \leq c\,\dist(X,Y)^2.
\end{equation}
Based on \eqref{eq:ApproxInvRetr}, we establish the counterpart of the equality \eqref{eq:TransInvExp} for a~general retraction.

\begin{lemma}\label{lem:TransInvRetr}
For any  $X$ and $Y$ in a~compact subset of the intersection of a~TNN and a~TRtrN
of $\tilde{X} \in \Spkn$, it holds that
\begin{equation}\label{eq:TransInvRetr}
\calR_Y^{-1}(X) = -{\calT}^{\,\|}_{XY}\calR_X^{-1}(Y) + O(\dist(X,Y)^2),
\end{equation}
where ${\calT}^{\,\|}_{XY}$ is the parallel transport along the geodesic connecting $X$ and $Y$.
\end{lemma}
\begin{proof}
Using \eqref{eq:ApproxInvRetr}, we can write
\begin{align}
    \calR_X^{-1}(Y) & = \exp_X^{-1}(Y) + RL_XY, \label{eq:TransInvRetr_proof1} \\
    \calR_Y^{-1}(X) & = \exp_Y^{-1}(X) + RL_YX, \label{eq:TransInvRetr_proof2}
\end{align}
with 
\begin{align}
RL_XY\in \romT_X\Spkn,  
    & \quad \|RL_XY\|_X = O\big(\dist(X,Y)^2\big), \label{eq:RLX}\\
RL_YX\in \romT_Y\Spkn,  
    & \quad \|RL_YX\|_Y = O\big(\dist(X,Y)^2\big). \label{eq:RLY}
\end{align}
The relations \eqref{eq:TransInvRetr_proof1} and \eqref{eq:TransInvRetr_proof2} together with \eqref{eq:TransInvExp} lead to
\[
\calR_Y^{-1}(X) = - \calT^{\,\|}_{XY}\exp_X^{-1}(Y) + RL_YX = -\calT^{\,\|}_{XY}\big(\calR_X^{-1}(Y) - RL_XY\big) + RL_YX.
\]
Then \eqref{eq:TransInvRetr} is obtained by using \eqref{eq:RLX}, \eqref{eq:RLY}, and the isometry of $\calT^{\,\|}_{XY}$.\qed
\end{proof}

Next, we derive the first-order approximation similar to
\eqref{eq:Taylor_VectorField} for an~inverse retraction.

\begin{lemma}
If $X$ belongs to a~compact subset of the intersection of a~TNN and a~TRtrN of $\tilde{X} \in \Spkn$, then it holds that
\begin{equation}\label{eq:Taylor_VectorField_Retr}
	\grad f(X) = {\calT}^{\,\|}_{\tilde{X}X}\grad f(\tilde{X}) 
        + {\calT}^{\,\|}_{\tilde{X}X}\Hess{\tilde{X}}{\calR^{-1}_{\tilde{X}}(X)} 
        + o\big(\big\|\calR^{-1}_{\tilde{X}}(X)\big\|_{\tilde{X}}\big).
		\end{equation}
\end{lemma}
\begin{proof}
In view of \eqref{eq:ApproxInvRetr} and \eqref{eq:Taylor_VectorField}, we obtain that
\begin{align*}
    \grad f(X) &= {\calT}^{\,\|}_{\tilde{X}X}\grad f(\tilde{X}) 
        + {\calT}^{\,\|}_{\tilde{X}X}\Hess {\tilde{X}}{\calR^{-1}_{\tilde{X}}(X)}\\
    &\quad - {\calT}^{\,\|}_{\tilde{X}X}\Hess{\tilde{X}}{RL_{\tilde{X}}(X)} 
        + o\big(\big\|\exp_{\tilde{X}}^{-1}(X)\big\|_{\tilde{X}}\big).
\end{align*}
Then \eqref{eq:Taylor_VectorField_Retr} holds due to the isometry of ${\calT}^{\,\|}_{\tilde{X}X}$, 
$\|RL_{\tilde{X}}(X)\|_{\tilde{X}} = O\big(\dist(\tilde{X},X)^2\big)$,
$\|\exp_{\tilde{X}}^{-1}(X)\|_{\tilde{X}} = \dist(\tilde{X},X)$, and \eqref{eq:DistRetr_Estimate_2}. \qed
\end{proof}

The next lemma can be viewed as an extension of~\cite[Lem.~3.1]{DembES1982} and~\cite [Lem.~7.4.8]{AbsiMS08}.
\begin{lemma}\label{lem:grad_dist}
Let $X_*\in\Spkn$ be such that $\grad f(X_*)=0$ and the Riemannian Hessian $\hess f(X_*)$ is nonsingular. Then, there are positive constants $c_3$, $c_4$, and $\delta_g$ such that for all \mbox{$X\in \Spkn$} with $\dist(X,X_*)<\delta_g$, it holds that
\begin{equation}\label{eq:Estimate_grad_dist}
c_3\,\dist(X,X_*) \leq \|\grad f(X)\|_{X} \leq c_4\,\dist(X,X_*).
\end{equation}
\end{lemma}
\begin{proof}
Let us set $\beta = \big\|\hess f(X_*)^{-1}\big\|_{X_*}$ and 
\begin{equation*}\label{eq:R(X)}
R(X) = \grad f(X)-{\calT}^{\,\|}_{X_*X}\grad f(X_*) - {\calT}^{\,\|}_{X_*X}\Hess{X_*}{\calR^{-1}_{X_*}(X)}.
\end{equation*} 
In view of \eqref{eq:Taylor_VectorField_Retr}, there exists $\delta_g \in (0,\delta_{X^*})$ such that for any $X\in\Spkn$ with $\dist(X,X_*)<\delta_g$, it holds that
$\|R(X)\|_{X}\leq \frac{1}{2\beta}\|\calR^{-1}_{X^*}(X)\|_{X_*}.$ Since $\grad f(X_*)=0$, we have 
\begin{equation*}
\grad f(X) = {\calT}^{\,\|}_{X_*X}\Hess{X_*}{\calR^{-1}_{X_*}(X)} + R(X).
\end{equation*}
Taking into account the fact that the parallel transport is isometric, it follows, on the one side, that
\begin{align}\notag
\big\|\grad f(X)\big\|_{X} &\leq \big\|\Hess{X_*}{\calR^{-1}_{X_*}(X)}\big\|_{X_*} + \big\|R(X)\big\|_{X}\\
& \leq \bigg(\big\|\hess f(X_*)\big\|_{X_*} + \frac{1}{2\beta}\bigg)\big\|\calR^{-1}_{X_*}(X)\big\|_{X_*}.
\label{eq:Dembo_proof_1}
\end{align}
On the other side, we obtain that
\begin{align}\notag
\big\|\grad f(X)\big\|_{X} &\geq \big\|\Hess{X_*}{\calR^{-1}_{X_*}(X)}\big\|_{X} 
- \big\|R(X)\big\|_{X}\\
& \geq \big\|\hess f(X_*)^{-1}\big\|_{X_*}^{-1}\big\|\calR^{-1}_{X_*}(X)\big\|_{X_*} - \frac{1}{2\beta} \big\|\calR^{-1}_{X_*}(X)\big\|_{X_*}\notag\\
&= \frac{1}{2\beta} \big\|\calR^{-1}_{X_*}(X)\big\|_{X_*}.
\label{eq:Dembo_proof_2}
\end{align}
Then the estimates \eqref{eq:Estimate_grad_dist} follow from \eqref{eq:DistRetr_Estimate_2}, \eqref{eq:Dembo_proof_1}, and \eqref{eq:Dembo_proof_2} with $c_3 = 1/(2\beta c_2)$ and $c_4 = \left(2\beta\|\hess f(X_*)\|_{X_*} + 1\right)/(2\beta c_1)$. \qed
\end{proof}

\paragraph{Superlinear convergence}
We are now ready to state the local convergence result for the Riemannian inexact Newton method given in Algorithm~\ref{alg:Newton_inexact}.
\begin{theorem}[Local superlinear convergence]\label{thm:LocalConvInexactNewton}
Assume that the initial guess \linebreak $X_0\in\Spkn$ is close enough to a~nondegenerate stationary point $X_*\in\Spkn$ of the cost function $f$, i.e., \mbox{$\grad f(X_*)=0$} and $\hess f(X_*)$ is nonsingular. Then, the sequence $\{X_j\}_j$ generated by Algorithm~\textup{\ref{alg:Newton_inexact}} is well-defined and converges to $X_*$ superlinearly.
\end{theorem}
\begin{proof}
First, we prove that for any $j\geq 0$, if $\tilde{Z}_j\in \TX$ is an~inexact solution to the Newton equation \eqref{eq:Newtoneq_general} satisfying 
\begin{equation}\label{eq:InexactSolving}
    \big\|\Hess{X_j}{\tilde{Z}_j} + \grad f(X_j) \big\|_{X_j} \leq \min\big\{\eta,\|\grad f(X_j)\|_{X_j}^\mu\big\}\big\|\grad f(X_j)\big\|_{X_j}
\end{equation}
with $0<\eta<1$, $\mu>0$ and $X_j$ close enough to $X_*$, then
\begin{equation}\label{eq:ConvInexactNew_concl}
    \lim\limits_{X_j\to X^*}\frac{\dist(\calR_{X_j}(\tilde{Z}_j),X_*)}{\dist(X_j,X_*)} = 0.
\end{equation}
Let $\res_j = \Hess{X_j}{\tilde{Z}_j} + \grad f(X_j)$ denote the residual of the inexact solver at $j$-th iterate.  Then, by assumption the initial guess $X_0$ can be chosen such that $\dist(X_0,X_*) <\delta \leq \min\{\delta_{\tilde{X}}, \delta_g, \tilde{r}\}$, where the constants are determined in \eqref{eq:hessnonsing}, \eqref{eq:DistRetr_Estimate} (with $\tilde{X}$ replaced by $X_*$), and \eqref{eq:Estimate_grad_dist}. Moreover, $\delta$ can be necessarily reduced to let \eqref{eq:Taylor_VectorField} hold,  $\|\tilde{Z}_0\|_{X_0}, \|\calR_{X_0}^{-1}(X_*)\|_{X_0} \in (0,\delta^{\tan}_{X_*})$, and $X_0$ 
belong to the intersection of a TNN and a TRtrN of $X_*$. Set $\beta=\|\hess f(X_*)^{-1}\|_{X_*}$. Following the same route as in the proof of \cite[Lem.~3.1]{BortFF22}, we obtain by using \eqref{eq:Retr_LipschitzConst}, \eqref{eq:TransInvRetr}, \eqref{eq:hessnonsing}, \eqref{eq:Taylor_VectorField_Retr}, and the isometry of the parallel transport that
\begin{align*}
&\dist\big(\calR_{X_j}(\tilde{Z}_j),X_*\big) 
= \dist\big(\calR_{X_j}(\tilde{Z}_j),\calR_{X_j}(\calR_{X_j}^{-1}(X_*))\big)\\
&\quad\leq K_{\calR,X_*}\big\|\hess f(X_j)^{-1}\big(\res_j-\grad f(X_j)\big) - \calR_{X_j}^{-1}(X_*)\big\|_{X_j}\\
&\quad= \!K_{\calR,X_*}\!\big\|\hess f(X_j)^{-1}\!\big(\res_j\!-\!\grad f(X_j)\big)\!\! +\! \calT^{\,\|}_{X_*X_j}\!\calR_{X_*}^{-1}\!(X_j)\! -\! O\big(\dist(X_*,\!X_j)^2\big)\!\big\|_{X_j}\\
&\quad\leq K_{\calR,X_*}\big\|\hess f(X_j)^{-1}\big\|_{X_j} \big\|\res_j - \grad f(X_j) + \Hess{X_j}{{\calT}^{\,\|}_{X_*X_j} \calR_{X_*}^{-1}(X_j)}\big\|_{X_j}\\ 
&\quad\quad+  K_{\calR,X_*}O\big(\dist(X_*,X_j)^2\big)\\
&\quad \leq 2\,\beta\, K_{\calR,X_*} 
\bigl\|\res_j + \big(\hess f(X_j){\calT}^{\,\|}_{X_*X_j}-{\calT}^{\,\|}_{X_*X_j} \hess f(X_*)\big)\calR_{X_*}^{-1}(X_j) \bigr.\\
 &\quad\quad- \bigl.o\big(\|\calR_{X_*}^{-1}(X_j)\|_{X_*}\big)\bigr\|_{X_j} + K_{\calR,X_*}O\big(\dist(X_*,X_j)^2\big)\\
&\quad\leq 2\,\beta\, K_{\calR,X_*}
\Big(\big\|\hess f(X_j){\calT}^{\,\|}_{X_*X_j}-{\calT}^{\,\|}_{X_*X_j}\hess f(X_*)\big\|_{X_j}\big\|\calR_{X_*}^{-1}(X_j)\big\|_{X_*} \Big.\\
&\quad\quad\Big.+ \big\|\res_j\big\|_{X_j} + o\big(\big\|\calR_{X_*}^{-1}(X_j)\big\|_{X_*}\big)\Big) + K_{\calR,X_*}O\big(\dist(X_*,X_j)^2\big)\\
&\quad\leq 2\,\beta\,K_{\calR,X_*}
\Bigl(\big\|{\calT}^{\,\|}_{X_j X_*}\hess f(X_j){\calT}^{\,\|}_{X_*X_j}-\hess f(X_*)\big\|_{X_j}\big\|\calR_{X_*}^{-1}(X_j)\big\|_{X_*} \Bigr.\\
&\quad\quad+ \Bigl.\big\|\res_j\big\|_{X_j} + o\big(\big\|\calR_{X_*}^{-1}(X_j)\big\|_{X_*}\big)\Bigr) + K_{\calR,X_*}O\big(\dist(X_*,X_j)^2\big).
\end{align*}
Then, \eqref{eq:ConvInexactNew_concl} is derived from the last inequality, \eqref{eq:PhessP},  \eqref{eq:DistRetr_Estimate}, \eqref{eq:InexactSolving}, and \eqref{eq:Estimate_grad_dist}.
			
The superlinear convergence follows from  \eqref{eq:ConvInexactNew_concl} and arguments similar to the proof of \cite[Thm.~3.1]{FernFY17}.\qed
\end{proof}
 
For the sake of independent interest, we restate Theorem~\ref{thm:LocalConvInexactNewton} in the general case as follows.
\begin{theorem}\label{thm:LocalConvInexactNewton_general}
Let $\Phi$ be a~continuously differentiable vector field on a Riemannian ma\-ni\-fold $\calM$ equipped with a Riemannian connection $\nabla$ and a retraction $\calR$. Assume that $x_*\in \calM$ is a~nondegenerate singular point of $\Phi$, i.e., $\Phi(x_*) = 0$, and the covariant derivative 
$\nabla \Phi(x_*)$ is nonsingular. Then, there exists $\delta >0$ such that for any starting point $x_0 \in \calM$ with $\dist(x_0,x_*)<\delta$, the iterate $x_{j+1} = \calR(\tilde{z}_j)$, where $\tilde{z}_j$, as a~tangent vector of $\calM$ at $x_j$, is an inexact solution to the Newton equation 
   \begin{equation}\label{eq:newtV}
       \nabla \Phi(x_j)z_j + \Phi(x_j) = 0
   \end{equation} 
satisfying $\|\nabla \Phi(x_j)\tilde{z}_j + \Phi(x_j)\|_{x_j} \leq \min\{\eta,\|\Phi(x_j)\|_{x_j}^\mu\}\|\Phi(x_j)\|_{x_j}$ with $ 0<\eta<1$ and $\mu>0$, 
is well-defined and superlinearly converges to $x_*$.
\end{theorem}

As in computation, the (exact) solution to the Newton equation \eqref{eq:newtV}
is usually intractable, Theorem~\ref{thm:LocalConvInexactNewton_general} is a~practical extension of \cite[Thm.~3.1]{BortFF22} to the inexact case.

%-----------------------------------------------------------------------------%
\subsection{Global convergence of the hybrid Riemannian Newton method}

Using the convergence results for the RGD method from \cite{GSAS21} and for Algorithms~\ref{alg:Newton_plain} and \ref{alg:Newton_inexact} as presented above, we establish the global convergence of Algorithm~\ref{alg:Newton_hybrid}.
  
\begin{theorem}\label{theo:Newton_hybrid}
Assume that $X_*$ is an accumulation point of the sequence $\{X_j\}_j$ ge\-ne\-rated by Algorithm~\textup{\ref{alg:Newton_hybrid}} with an appropriately chosen \emph{switching parameter} $\theta$ and that $\hess f(X_*)$ is nonsingular. Then $X_*$ is a critical point of the minimization problem~\eqref{eq:opt_prob}. Furthermore, $\{X_j\}_j$ converges quadratically {\rm(}resp. superlinearly{\rm)} to $X_*$ if Algorithm~\textup{\ref{alg:Newton_plain}} {\rm(}resp. Algorithm~\textup{\ref{alg:Newton_inexact}}{\rm)} is adopted. 
If $\hess f(X_*)$ is additionally positive-definite, then $X_*$ is a~minimizer of $f$.
\end{theorem}
\begin{proof}
It has been proven in \cite[Thm.~5.7]{GSAS21} that any accumulation point $X_*$ of the RGD method applied to \eqref{eq:opt_prob} is a~critical point, i.e., \mbox{$\grad f(X_*) = 0$}. That is, if the switching parameter $\theta$ in Algorithm~\ref{alg:Newton_hybrid} is chosen small enough, thanks to Lemma~\ref{lem:grad_dist}, $X_{j_*}$ generated in the first phase is already close enough to $X_*$ for some~$j_*$. Starting the second phase with $X_{j_*}$,  Algorithm~\ref{alg:Newton_plain} generates a~sequence $\{X_j\}_j$ which, in view of \cite[Thm.~6.3.2]{AbsiMS08} or \cite[Thm.~6.7]{Boumal23}, converges to $X_*$ quadratically.
The superlinear convergence in case of using Algorithm~\ref{alg:Newton_inexact} in Algorithm~\ref{alg:Newton_hybrid} follows immediately from Theorem~\ref{thm:LocalConvInexactNewton}.
The last statement is a~consequence of the sufficient second-order optimality conditions. \qed
\end{proof}

 %------------------------------------------------------------------
\section{Numerical examples}
\label{sec:Numer}
We will test the proposed optimization schemes and compare them with the existing methods on several problems.
For each example of small dimension, we present the results of the RGD method  \cite[Alg.~1]{GSAS21} and the hybrid Riemannian Newton methods with the second phase using the Riemannian Newton method (hRN) as in Algorithm~\ref{alg:Newton_plain} or using the Riemannian inexact Newton  (hRiN) method as in Algorithm~\ref{alg:Newton_inexact}, where MINRES is employed in the inner iteration. For large problems, the hRN method is, however, excluded. 

The two retractions presented at the beginning of section~\ref{sec:Solve_NewtonEq} will be used, which results in different schemes whose names will be made up by adding ``Cay'' for the Cayley retraction and ``SR'' for the SR retraction.
Furthermore, we consider different metrics and extend the names of the corresponding optimization schemes by ``c'' for the canonical-like metric, ``e'' for the Euclidean metric, and ``M'' for the weighted Euclidean metric with a~suitably chosen weighting matrix $\Mb$. For models that have the Euclidean Hessian of the form $\nabla^2\bar{f}(X)[Z] = MZ$ with an~spd matrix $M$, we will choose the weighted Euclidean metric $g_{\Mb}\!$ with $\Mb=M$.

For the RGD method combined with non-monotone line search~\cite[Alg.~1]{GSAS21} and the RGD phase in the hybrid Riemannian Newton methods, we use the parameters $\alpha = 0.85$, $\beta = 1\e-4$, $\delta = 0.5$, $\gamma_0 = 1\e-3$,
$\gamma_{\min} = 1\e-15$, and $\gamma_{\max} = 1\e+5$, if not specified otherwise. 
In hRiN, the inexact Newton parameters are set to $\eta = 1\e-3$, $\mu=0.5$, and the maximal number of inner iterations in MINRES is chosen as~$nk$. 
 
 Our initial tests showed that the second phase of the hybrid Newton schemes can always be run with the unit step size provided that the switching parameter~$\theta$ was chosen small enough. This results, however, in much more iterations in the first phase. To overcome this difficulty, in our implementation, we invoke in the second phase the so-called damping strategy, in which the step size is determined adaptively by using the backtracking linear search which guarantees a~monotone reduction in the value of the cost function. We use the same parameters as in the first phase except for $\alpha=0$ and a~smaller $\delta = 0.2$ to possibly reduce the number of backtracking steps. To avoid the dependence on the quality of the initial guess and the choice for metric, we will consider the method convergent at step $j$ if within a~given number of iterations \texttt{mxit}, 
 the condition 
 \begin{equation}\label{eq:stop}
 \|\grad f(X_j)\|_{X_j} \leq \texttt{tol}\|\grad f(X_0)\|_{X_0}
 \end{equation}
 is fulfilled for a~given tolerance ${\tt tol}$. In this case, we set $X_* = X_j$ and  $f_* = f(X_*)$. If the tolerance inequality does not hold, $X_*$ is the last iterate. The switching scheme is controlled by $\theta$ in a~similar way. 

All computations are done on a~standard laptop with an Intel(R) Core(TM) i7-4500U CPU at 1.80~GHz (up to 2.40~GHz) and 8~GB of RAM running MATLAB 2023a under Windows~10 Home.

%---------------------------------------------------------
\subsection{Symplectic solution of a~matrix least squares problem}
Given a nonsingular matrix $A\in \Rbnn$ and $B\in \Rbnk$, the matrix equation 
\begin{equation}\label{eq:MatEq}
			AX = B
\end{equation}
has a unique solution $X = A^{-1}B$. We are now aiming to find a~solution to \eqref{eq:MatEq} in the class of symplectic matrices. In general, neither the existence of a solution to this constrained equation nor its uniqueness, if it exists, is guaranteed. A~natural idea is to solve this problem in the least squares sense by minimizing the residual on the symplectic Stiefel manifold $\Spkn$, i.e.,
\begin{equation}\label{eq:costfunc_nearestprob_3}
	\min_{X\in\Spkn} f(X) = \frac{1}{2}\|AX-B\|_{\mathrm{F}}^2. 
\end{equation}
The ambient gradient and Hessian of the cost function in \eqref{eq:costfunc_nearestprob_3} are given by
\[
    \nabla\bar{f}(X) = A^T\!AX-A^TB\quad\text{and}\quad \Nablaf{X}{Z} = A^T\!AZ,
\]
respectively. Note that in a~special case, when $A \in \Spn$ and $B \in \Spkn$, the matrix equation~\eqref{eq:MatEq} and the optimization problem~\eqref{eq:costfunc_nearestprob_3} have a~unique solution $X_{\min} = J_{2n}^TA^T\!J_{2n}B$. 

In the first test, we consider a~small problem for the availability of the hRN method,  where $A \in \Spn$ and $B \in \Spkn$ with $n=50$ and $k = 6$ are generated as follows. For the coefficient matrix $A$, after setting random generator as \texttt{default} (which is indeed \texttt{rng(0,`twister')}), we first generate the matrices $\tilde{A}_1 = \texttt{rand}(n,n)$, $\tilde{A}_2 = \texttt{rand}(n,n)$ and then symmetrize them by setting $A_1 = 0.1(\tilde{A}_1^{} + \tilde{A}_1^T)$ and $A_2 = 0.1(\tilde{A}_2^{} + \tilde{A}_2^T)$. At the end, we choose the coefficient matrix 
$A = \left[\begin{smallmatrix}
			I_n&A_1\\A_2&C
\end{smallmatrix}\right]$ with $C = I_n + A_2A_1$ as in~\cite{DopiJ09}. The right-hand side $B$ and the initial guess $X_0$ are chosen as the randomly generated matrices from \texttt{\texttt{rng(1,`twister')}} and \texttt{\texttt{rng(0,`philox')}}, respectively, and then simplecticized using the SR decomposition~\cite{GSS24}. For this small-size problem, we set $\texttt{tol}=1\e-10$, and for all hR(i)N schemes, we choose $\theta = 1\e-4$ except for $\theta = 1\e-5$ for the schemes under the Euclidean metric. 

Detailed results for the convergent schemes are presented in Table~\ref{tab:matrixeq}, where ``\#iter'', ``time'' and ``feas'' stand for the number of (outer) iterations, the wall-clock time in seconds and the feasibility $\|X_*^TJ_{2n}X_*^{}-J_{2k}\|_{\mathrm{F}}$. The convergence history of the runs is given in Figure~\ref{fig:matrixeq_small}. The following observations can be easily drawn from these results: 1)~only the schemes under the carefully chosen metric converge and/or they converge much faster; 2)~for the reasonable switching parameter, $\theta = 1\e-4$ in this test, the results of the hRiN methods are almost a copy of that of the corresponding hRN methods except for the fact that they are computationally less expensive in the second phase, as expected; however, if a~small switching parameter is used, e.g., $1\e-5$, hRN methods are more accurate and expensive than the hRiN methods; 3)~the combination of the SR retraction and the weighted Euclidean metric delivers the best results, which also validates the preconditioning effect. In addition, with the same setting but without fixing the generator and seed in generating the starting point $X_0$, i.e., it is generated \linebreak

\begin{table}[hb]
		\centering
		\small
		{\caption{Symplectic least squares problem: $n=50, k = 6$, $\texttt{tol} = 1\e-10$, $\theta=1\e-5$ for the Newton methods under the Euclidean metric and $\theta=1\e-4$ for the rest.}
		\label{tab:matrixeq}}
		\begin{tabular}{lrrrrcccc}
			\toprule
			 \multirow{2}{*}{Method} & \multicolumn{2}{c}{\#iter} & \multicolumn{2}{c}{time} & \multirow{2}{*}{$f_*$} & \multirow{2}{*}{$\|\grad f(X_*)\|_{X_*}$} & \multirow{2}{*}{$\frac{\|X_*-X_{\min}\|_{\mathrm{F}}}{\|X_{\min}\|_{\mathrm{F}}}$} & \multirow{2}{*}{feas}\\\cmidrule(l){2-3} \cmidrule(l){4-5} 
    &\quad 1st&2nd&\quad 1st&2nd&&&&\\ \midrule
RGD-Cay-M  & 495   &&   0.86      &&  $3.2\e-11$&  $2.3\e-08$ &  $1.4\e-07$ & $8.3\e-05$\\
  RGD-SR-M  &  40 && 0.06 && $4.9\e-21$ & $9.9\e-11$ & $2.3\e-12$ & $3.6\e-12$ \\\midrule
 hRN-Cay-e &  2948 &   4 & 1.56 & 0.38 & $2.8\e-23$ & $1.9\e-11$ & $1.6\e-13$ & $2.7\e-11$\\
  hRN-SR-e  & 2916 &  4 & 2.32 & 0.36 & $7.8\e-23$ & $1.4\e-10$ & $5.0\e-13$ & $3.1\e-12$\\
 hRN-Cay-M & 494 & 1 & 0.89 & 0.09 & $3.2\e-11$ & $1.5\e-08$ & $1.4\e-07$ & $8.3\e-05$\\
  hRN-SR-M & 38 & 2 & 0.05 & 0.19 & $1.7\e-22$ & $1.8\e-11$ & $7.3\e-13$ & $3.4\e-12$\\\midrule
hRiN-Cay-e & 2948 & 6 & 1.51 & 0.11& $7.3\e-15$ & $3.9\e-09$ & $8.1\e-09$ & $1.9\e-06$\\
 hRiN-SR-e & 2916 & 5 & 2.66 & 0.09 & $7.5\e-17$ & $2.0\e-09$ & $1.1\e-09$ & $3.0\e-12$\\
hRiN-Cay-M & 494 & 1 & 0.92 & 0.01 & $3.2\e-11$ & $1.5\e-08$ & $1.4\e-07$ & $8.3\e-05$\\
 hRiN-SR-M & 38 & 2 & 0.06 & 0.02 & $9.4\e-23$ & $1.4\e-11$ & $3.5\e-13$ & $3.2\e-12$\\
\bottomrule
\end{tabular}
\end{table}

\begin{figure}[t]
 		\centering
  \input{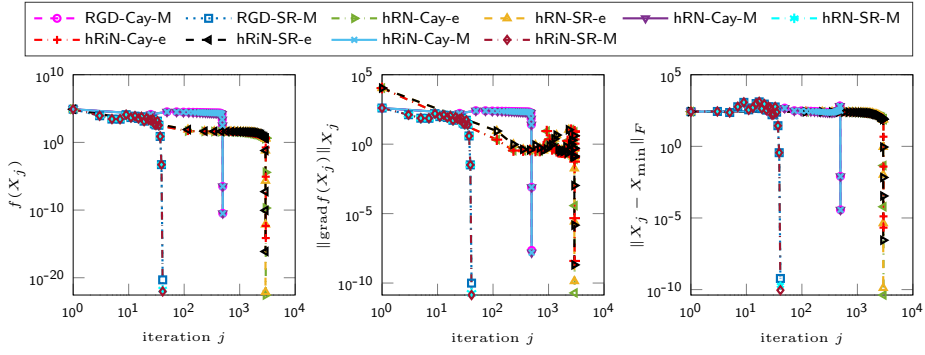}
 		{\caption{Symplectic solution of a matrix least squares problem with $n=50, k = 6$ }
 			\label{fig:matrixeq_small}}
\end{figure}

\noindent
randomly and differently every run, we run again hRN-Cay-e and hRiN-SR-M schemes 50 times. The resulting iterate always tends to the minimizer which arguably means that the hybrid Newton scheme converges globally.

In the second test, we invoke some sparse matrices in the MATLAB gallery as \texttt{A1 = 0.5*gallery(`poisson',20), A2 = 0.1*gallery(`tridiag',400)} and construct $A$ as above; $B$ and $X_0$ are chosen similarly except for the initializations by \texttt{default} and \texttt{rng(1,`twister')}, respectively, and $k=10$. Moreover, we set the maximal number of iterations $\texttt{maxit} = 5000$, $\texttt{tol} = 1\e-8$, and $\theta = 1\e-2$ and $\theta = 1\e-3$ for the hRiN methods under the weighted and standard Euclidean metrics, respectively. 

Numerical results for all convergent combinations of the optimization methods, retractions, and metrics, can be found in Table~\ref{tab:matrixeq_sparse}; the running history is given in~Figure~\ref{fig:matrixeq_sparse_new}. One can observe that the weighted Euclidean metric improves the accuracy of the solution by delivering smaller values and gradient norms of the cost function. It should also be noted that the hRiN schemes in the second phase do not help considerably accelerate the solution computation. The reason is that there are too few Newton steps in the iterates. One can try to increase the switching parameter~$\theta$ to ensure an earlier switch to the second phase. 
%include more of them but then can not be sure 
In this case, however, it cannot be guaranteed that the whole optimization process will be faster, since a~Newton step is computationally more expensive than a~gradient descent step. 

\begin{table}[hb]
	\centering
	\small
	{\caption{Symplectic least squares problem: $n=400$, $k = 10$, $\texttt{tol} = 1\e-6$, $\theta = 1\e-2$ for hRiN-Cay/SR-M and $\theta = 1\e-3$ for hRiN-Cay/SR-e.} \label{tab:matrixeq_sparse}}
	\begin{tabular}{lrrrrcccc}
        \toprule
	\multirow{2}{*}{Method} & \multicolumn{2}{c}{\#iter} & \multicolumn{2}{c}{time} & \multirow{2}{*}{$f_*$} & \multirow{2}{*}{$\|\grad f(X_*)\|_{X_*}$} & \multirow{2}{*}{$\frac{\|X_*-X_{\min}\|_{\mathrm{F}}}{\|X_{\min}\|_{\mathrm{F}}}$} & \multirow{2}{*}{feas}\\\cmidrule(l){2-3} \cmidrule(l){4-5} 
    &\quad 1st&2nd&\quad 1st&2nd&&&&\\ \midrule
%RGD-Cay-c & 2000 && 279.0 && $4.7\e+03$& $5.4\e+02$ &$9.8\e-01$&$2.0\e+00$\\
%RGD-SR-c & 2000 && 18.4 && $1.5\e+03$ & $1.3\e+01$& $8.5\e-01$& $3.0\e-12$\\
RGD-Cay-e & 755 && 2.91 && $3.6\e-10$ & $6.1\e-06$ & $5.9\e-07$& $5.3\e-11$\\
RGD-SR-e & 343 && 1.49 && $2.0\e-10$ & $4.8\e-06$ & $4.3\e-07$ & $3.1\e-12$\\
RGD-Cay-M & 239 && 4.48 && $7.0\e-16$ & $1.2\e-12$ & $3.4\e-10$ & $1.1\e-06$\\
RGD-SR-M & 69 && 1.24 && $8.6\e-15$ & $1.3\e-07$ & $1.2\e-09$& $2.8\e-12$\\\midrule
hRiN-Cay-e & 388 & 60 & 1.84 & $6.70$ & $1.3\e-11$ & $5.5\e-06$ & $8.3\e-08$ & $5.4\e-11$\\
hRiN-SR-e & 165 & 3 & $0.71$ & $0.37$ & $3.0\e-14$ & $7.1\e-08$ & $4.8\e-09$ & $2.2\e-12$ \\
hRiN-Cay-M & 237 & 2 & 4.63 & 0.11 & $7.0\e-16$ & $1.2\e-12$ & $3.4\e-10$ & $1.1\e-06$\\
hRiN-SR-M & 67 & 2 & 1.22 & 0.11 & $8.7\e-15$ & $1.3\e-07$ & $1.2\e-09$ & $2.5\e-12$ \\
\bottomrule
\end{tabular}
	\end{table}

 \begin{figure}[t]
 		\centering
   \input{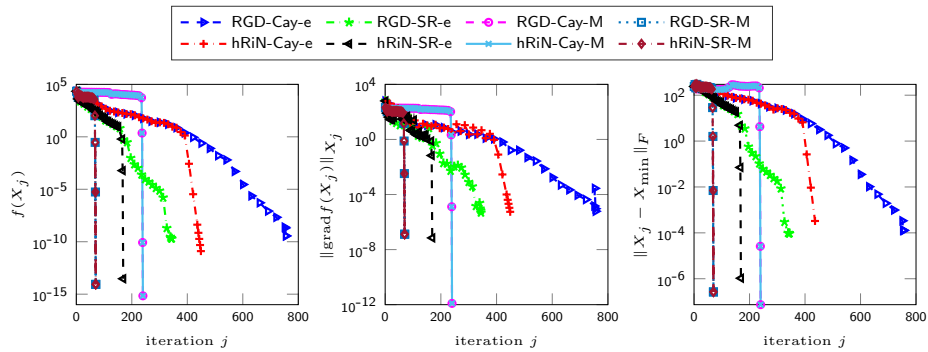}
 		{\caption{Symplectic solution of a matrix least squares problem with $n=400, k = 10$ }
 			\label{fig:matrixeq_sparse_new}}
 \end{figure}

%----------------------------------------------------------
\subsection{Symplectic trace minimization}
Next, we consider the trace minimization problem  over a symplectic Stiefel manifold
\begin{equation}\label{eq:cost_trace}
\min_{X\in \Spkn} f(X) = \frac{1}{2}\tr(X^TAX) 
\end{equation}
with an~spd matrix $A\in\R^{2n\times 2n}$. This minimization problem is essential in computing the smallest symplectic eigenvalues of $A$, see \cite{SonAGS21} and references therein. The ambient gradient and Hessian of the function in~\eqref{eq:cost_trace} have the form
\[
\nabla\bar{f}(X) = AX \quad\text{and}\quad \Nablaf{X}{Z}=AZ,
\]
respectively. 

In the first test, we construct $A = S^T\tilde{D}S$ with $\tilde{D}=\diag(1, 2, \ldots, n, 1, 2, \ldots, n)$ and 
\mbox{$S = \left[\begin{smallmatrix}
I_n&S_1\\S_2&I_n +S_2S_1
\end{smallmatrix}\right]$}, where $S_1$ and $S_2$ are the sparse random symmetric matrices gene\-rated  as \mbox{$S_1 = \texttt{sprandsym}(n,3/n,0.1,1)$} and $S_2 = \texttt{sprandsym}(n,3/n,0.01,1)$. Since the symplectic spectrum is symplectic invariant \cite[Prop.~8.14]{deGo06}, the symplectic eigenvalues of $A$ are just $1, 2,\ldots,n$ and, therefore, the cost function in \eqref{eq:cost_trace} has the minimal value $f_{\min} = f(X_{\min})=1+\cdots+k$. 
The initial guess is taken as $X_0=\left[\begin{smallmatrix}
    I_{n,k}&0_{n,k}\\0_{n,k}&I_{n,k}
\end{smallmatrix}\right]$ with $I_{n,k}=[I_k\;\; 0]^T\in\R^{n\times k}$. 
Other parameters are set to $\texttt{mxit} = 2000$, $\texttt{tol} = 1\e-8$, and $\theta = 1\e-3$. With this setting, all tested schemes are convergent. Numerical results are given in Table~\ref{tab:sympltracemin_new}, while the history of the gradient norm and the distance to $f_{\min}$ are reported in Figure~\ref{fig:tracemin_new}.

\begin{table}[tbp]
\centering
\small
{\caption{\label{tab:sympltracemin_new} Symplectic trace minimization with synthetic data: $n=2000$, $k = 5$, $\texttt{tol} = 1\e-8$, $\theta = 1\e-3$.}} 
			
\begin{tabular}{lrrrrccc}
    \toprule
	\multirow{2}{*}{Method} & \multicolumn{2}{c}{\#iter} & \multicolumn{2}{c}{time} & \multirow{2}{*}{$f_*$} & \multirow{2}{*}{$\|\grad f(X_*)\|_{X_*}$}  & \multirow{2}{*}{feas}\\\cmidrule(l){2-3} \cmidrule(l){4-5} 
    &\quad 1st&2nd&\quad 1st&2nd&&&\\ \midrule
  RGD-Cay-c&982&&4.40&&$1.3\e-10$& $1.6\e-05$ &$1.4\e-13$\\
  RGD-SR-c&1126&&5.89&&$5.3\e-11$&$1.0\e-05$& $8.3\e-16$\\
  RGD-Cay-e&1279&&5.96&&$1.7\e-10$&$1.6\e-05$&$1.6\e-13$\\
  RGD-SR-e&1451&&6.84&&$5.8\e-11$&$1.1\e-05$ & $1.5\e-15$\\
  RGD-Cay-M&19&&0.54&&$6.8\e-14$&$1.8\e-07$&$2.6\e-14$\\
  RGD-SR-M&17&&0.49&&$5.3\e-15$&$2.7\e-07$&$5.2\e-16$\\\midrule
  hRiN-Cay-e&100&4\;&0.44&2.80&$1.3\e-12$&$4.0\e-07$&$3.2\e-13$\\
  hRiN-SR-e&93&4\;&0.48&2.64&$3.2\e-14$&$8.6\e-08$&$7.1\e-16$\\
  hRiN-Cay-M&9&2\;&0.23&26.3&$1.1\e-13$&$3.6\e-08$&$3.2\e-14$\\
  hRiN-SR-M&9&2\;&0.27&0.40&$6.6\e-14$&$1.1\e-07$&$9.2\e-16$\\
			\bottomrule
		\end{tabular}
\end{table}

\begin{figure}[tbp]
	\centering
   \input{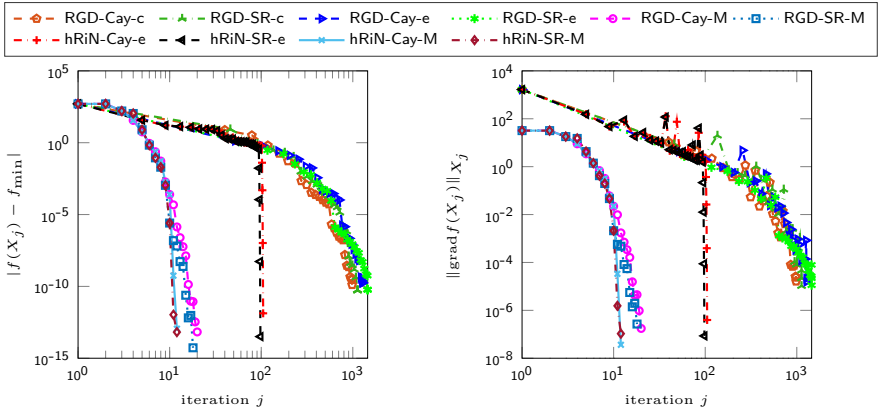}
			{\caption{Symplectic trace minimization with synthetic data: $n=2000$, $k = 5$, $\texttt{tol} = 1\e-8$, $\theta = 1\e-3$}
				\label{fig:tracemin_new}}
\end{figure}

Considering the same minimization problem, but this time with the spd matrix $A=J_{2n}H$, where $H$ is a~Hamiltonian matrix originating from the simulation of a wire saw model \cite{WeiK00} which is  a~weakly damped gyroscopic system. The smallest symplectic eigenvalues of $J_{2n}H$ help to analyze the stability of this system. The data are generated similarly as in \cite[Sect.~6]{SonAGS21} followed by a~normalization $JH/\|JH\|_{\mathrm{F}}$. As a~result, we obtain a minimization problem in $\Spkn$ with $n=2000$ and $k=5$. For the initial guess, we first take a random $2n\times 2k$ matrix with the {\tt default} generator and then choose $X_0$ as the symplectic factor in its SR decomposition. The key differences from the previous example are that $A$ is (almost) dense, which makes the inverse in some formulations rather expensive, and that a~minimizer or a~minimal value of the cost function is unknown. As a~consequence, the preconditioned version of the RGD and Riemannian Newton methods, whose preconditioner is presumably chosen to be $A$ itself, are time-consuming. Meanwhile, RGD methods under the Euclidean metric do not provide a~good enough initial guess for the Riemannian Newton method. 
Therefore, in Figure~\ref{fig:tracemin_wiresaw}, we only report the results of convergent runs, namely, RGD-Cay-c, RGD-SR-c, RGD-Cay-M and RGD-SR-M with the setting $\texttt{tol} = 1\e-8$ and $\theta = 1\e-3.$ One can observe that the two preconditioned RGD schemes reach the tolerance in a~few steps. These two methods also deliver the smallest cost value although they are more expensive to obtain, see Table~\ref{tab:tracemin_wiresaw} for detail. It is however worth noting that, although faster, the RGD methods under the canonical-like metric stagnate when $\|\grad f(X_j)\|_{X_j}/\|\grad f(X_0)\|_{X_0} \approx 1\e-10$, while that of both RGD-Cay-M and RGD-SR-M can reach $\approx 1\e-12$. Indeed, the result from RGD-SR-M is used as a~reference solution which appears in the right plot in Figure~\ref{fig:tracemin_wiresaw}; it reveals how fast the cost values of different methods approach the reference value $f_{\rm{ref}}$.

\begin{table}[htbp]
		\centering
		\small
		{\caption{Symplectic trace minimization with data from a wire saw model: $n=2000$, $k = 5$, $\texttt{tol} = 1\e-8$.
  \label{tab:tracemin_wiresaw}}}
		\begin{tabular}{lrrccc}
			\toprule
			 Method & \#iter & \quad time & $\qquad f(X_*)\qquad$ & $\|\grad(X_*)\|_{X_*}$& feas\\\midrule
  RGD-Cay-c & 179 & 10.8 &$1.02\e-05$ &$1.27\e-05$ &$1.50\e-08$\\
  RGD-SR-c & 200 & 11.9 & $1.06\e-05$ & $1.26\e-05$ & $1.64\e-15$\\
  RGD-Cay-M & 27 & 157 & $1.19\e-07$ & $1.02\e-07$ & $2.06\e-11$\\
  RGD-SR-M & 22 & 131 & $1.19\e-07$ & $8.54\e-08$ & $1.44\e-15$\\\midrule
  %RGD-Cay-c & 5000 & 394 &$6.83\e-07$ & $4.17\e-07$ & $1.5\e-08$\\
  %RGD-Cay-e & 5000 & 541 &$2.78\e-05$ & $6.16\e-06$ & $2.6\e-10$\\
  %RGD-SR-e & 5000 & 366 & $1.32\e-05$ & $4.09\e-07$ & $2.9\e-15$\\
  RGD-SR-M(ref.) & 41 & 236 & $1.19\e-07$ & $6.17\e-12$ & $2.22\e-15$\\
			\bottomrule
		\end{tabular}
	\end{table}

\begin{figure}[htbp]
	\centering
   \input{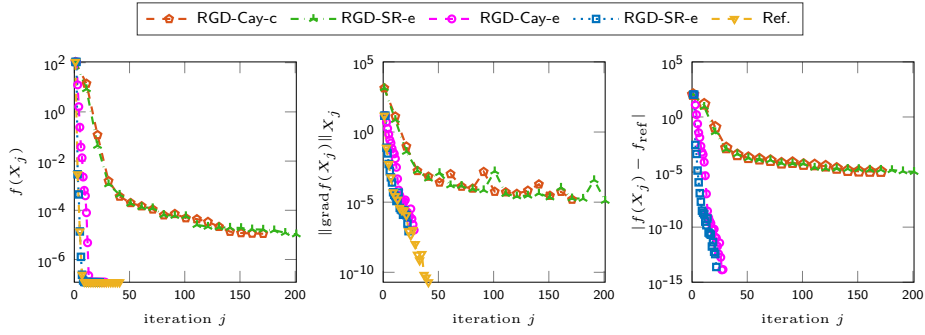}
			{\caption{Symplectic trace minimization with data from a wire saw model: $n=2000$, $k = 5$, $\texttt{tol} = 1\e-8$, $\theta = 1\e-3$} \label{fig:tracemin_wiresaw}}
\end{figure}

%------------------------------------------------------------------
\subsection{Trace minimization of a fourth-order function}
Consider a minimization problem 
\begin{equation*}
    \min_{X\in \textrm{Sp}(2n)} f(X) = \frac{1}{2}\tr(X^TAXX^TBX)
\end{equation*}
with the spd matrices $A, B\in\R^{2n\times 2n}$. The ambient gradient and Hessian of the cost function are given by
\begin{align*}
\nabla \bar{f}(X) & = BX(X^TAX) + AX(X^TBX), \\
\Nablaf{X}{Z} &= BZ X^TAX + AZ X^TBX + BXZ^TAX + BX(AX)^TZ\\
&\quad + AX(BX)^TZ + AXZ^TBX,
\end{align*}
respectively. As the Euclidean Hessian of $f$ is cumbersome, an efficient preconditioner can hardly be found. We therefore consider only running the RGD-SR-e and the hRiN-SR-e methods. 

We take the dimension $n=20$ and construct $A = A_1^TA_1^{}$ and $B = B_1^TB_1^{}$, where both $A_1,B_1\in\R^{2n\times 2n}$ are randomly generated of type \texttt{twister} with seed $0$ and $2$, respectively; after that follows a normalization of them. The initial guess is chosen as $X_0 = \left[\begin{smallmatrix}
Q_1&0\\0&Q_2
\end{smallmatrix}\right]$, where $Q_1, Q_2\in\R^{n\times n}$ are randomly generated of type \texttt{twister} with seed 1, followed by orthogonalizations. With the stopping tolerance $\texttt{tol} = 1\e-11$, the switching parameter $\theta = 1\e-7$, only the RGD-SR-c and hRiN-SR-e schemes convergence, where the latter is slightly faster and more accurate than the earlier in the sense that the resulting gradient norm of the cost function is smaller. Note also that after $\texttt{mxit} = 5000$ iterations, the RGD-SR-e scheme is almost convergent with the cost function value of the same order of but a~bit larger gradient norm. More details can be found in Table~\ref{tab:tracemin4}.

\begin{table}[tbp]
\centering
\small
{\caption{\label{tab:tracemin4} Symplectic trace minimization of a fourth-order function: $n= k = 20$, $k = 5$, $\texttt{tol} = 1\e-11$, $\theta = 1\e-7$. }}	
\begin{tabular}{lrrrrccc}
	\toprule
    \multirow{2}{*}{Method} & \multicolumn{2}{c}{\#iter} & \multicolumn{2}{c}{time} & \multirow{2}{*}{$f_*$} & \multirow{2}{*}{$\|\grad f(X_*)\|_{X_*}$} & \multirow{2}{*}{feas}\\\cmidrule(l){2-3} \cmidrule(l){4-5} 
    &\quad 1st&2nd&\quad 1st&2nd&&\\ \midrule
  RGD-SR-c & 2766 && 14.8 && $6.71\e-03$ & $1.00\e-11$ & $3.96\e-15$\\
  RGD-SR-e & 5000 && 31.6 && $6.71\e-03$&$5.09\e-11$& $4.05\e-16$\\
  hRiN-SR-e & 2270 & 4 & 9.48 & 2.45 & $6.71\e-03$ & $8.79\e-13$ & $3.92\e-15$\\
			\bottomrule
		\end{tabular}
\end{table}

\subsection{Discussion}
First and most importantly, in most cases, all the proposed methods work as expected. Acting as a preconditioning step, a~good choice of the metric, if available, considerably accelerates the RGD method and makes it very competitive. Furthermore, with an appropriate value of the switching parameter $\theta$, the Riemannian Newton phase in the hybrid schemes helps to push the iteration faster to the solution with much fewer iterates and less time. In each example, either or both the proposed methods yield faster and more accurate solutions. Second, the Riemannian inexact Newton method, if converging, can replace the Riemannian Newton method with an ignorable reduction in the solution quality. Third, for medium-to-large scale problems, especially when the matrix that represents the chosen metric is dense, using a weighted Euclidean metric has to be carefully considered because it will considerably slow down the loop due to the matrix inverse. Finally, the lack of an efficient solver for the Riemannian Newton equation also makes the method less attractive.

\section{Conclusion}\label{sec:Conclusion}
We have constructed a general Riemannian geometry of the symplectic Stiefel manifold under a family of tractable metrics which provides us freedom to choose the metric for a preconditioning of the problem. The framework suits extremely well the cost functions with constant Euclidean Hessian. Using the approach for tractable metrics, we explicitly computed the Riemannian Hessian of the cost function. Then we have calculated the solution/inexact solution to the Newton equation and constructed the Riemannian (inexact) Newton algorithm on its own and as the second phase of a hybrid Riemannian (inexact) Newton method. We have also proved the global convergence of our hybrid algorithm as well as established the local convergence rate. 

The presented numerical examples showed on the one hand that our proposed methods in most cases work well and outperform the plain Riemannian gradient method. But on the other hand, they pose some issues that might be addressed in a~future study. A constant Euclidean Hessian is quite a strict condition to practice the preconditioning by choosing a metric; therefore, an efficient way of approximating the Euclidean Hessian can enlarge the application of this method. Solving the Newton equation is an important step in practicing the Riemannian (inexact) Newton method but so far, an~efficient solver is still lacking.

%-----------------------------------------------------------------------
\appendix
\section{Proof of Theorem~\ref{th:Rhess_canon-like}}\label{appendix:proof_Prop_Rhess_canon-like}
For computation purpose, we rewrite the Riemannian Hessian in \eqref{eq:Rhessian} as
\begin{align} 
   \Hessc{X}{Z} &= \projc\Bigl(\Mb_{X,c,\rho}^{-1}\Nablaf{X}{Z} 
        + (\D_Z^{} \projc)\Mb_{X,c,\rho}^{-1}\nabla \bar{f}(X) \nonumber\\ 
		&\quad- \Mb_{X,c,\rho}^{-1}(\D_Z^{} \Mb_{X,c,\rho})\Mb_{X,c,\rho}^{-1}\nabla \bar{f}(X) \nonumber  \\
        & \quad + \Mb_{X,c,\rho}^{-1}\mathcal{K}\big(Z, \projc(\Mb_{X,c,\rho}^{-1}\nabla \bar{f}(X))\big)\Bigr). \label{eq:Rhessian_rewrite}
	\end{align}
	To make a step forward to the detailed formulation of this Hessian, by using \eqref{eq:B_X_E_inverse} and \eqref{eq:K_detail}, we first calculate 
	\begin{align}\notag
		\Mb_{X,c,\rho}^{-1}\mathcal{K}(Z, U)= &\ XJ_{2k}\skew(Z^TJ_{2n}^T U) + 2\,XX^TJ_{2n}\sym(Z U^T)J_{2n}X\\\notag
		&-\rho\, X\sym(Z^T\PXp U) + \frac{2}{\rho}\,J_{2n}^{}X_\perp^{} X_\perp^T\sym(Z U^T)J_{2n}X\\ \notag
		&-J_{2n}X_\perp^{} X_\perp^TJ_{2n}^TX(X^TX)^{-1}\sym(Z^T\PXp U)\\
		\notag
		&-J_{2n}X_\perp^{} X_\perp^TJ_{2n}^T\PXp Z\skew\big((X^TX)^{-1}X^T U\big)\\ 
		& - J_{2n}X_\perp^{} X_\perp^TJ_{2n}^T\PXp   U\skew\big((X^TX)^{-1}X^TZ\big). \label{eq:MinvK}
	\end{align}
Next, we plug \eqref{eq:Rgrad_canon} into \eqref{eq:MinvK} to get $\Mb_{X,c,\rho}^{-1}\mathcal{K}\big(Z,\projc(\Mb_{X,c,\rho}^{-1}\nabla\bar{f}(X))\big)$ term by term.
	The first term in \eqref{eq:MinvK} is a~normal vector, c.f. \eqref{eq:normal_space_canonical}, and, hence, it vanishes under the act of~$\projc$ in \eqref{eq:Rhessian_rewrite}. For the second term, we have
	\begin{align}\nonumber
		2\,XX^T&J_{2n}\sym\big(Z\big(\projc\big(\Mb_{X,c,\rho}^{-1}\nabla\bar{f}(X)\big)\big)^T\big)J_{2n}X\\
		%& = XX^TJ_{2n}\Bigl(\rho Z\nabla\bar{f}(X)^TXX^T + \rho XX^T\nabla\bar{f}(X) Z^T\\
		%&\quad + Z\skew(\rho J_{2k}^TX^{T}\nabla\bar{f}(X))J_{2k}^TX^T - XJ_{2k}\skew(\rho J_{2k}^TX^{T}\nabla\bar{f}(X)) Z^T\Bigr)J_{2n}X\\
		%&=\rho XX^TJ_{2n} Z\nabla\bar{f}(X)^TXJ_{2k} + \rho XJ_{2k}X^T\nabla\bar{f}(X) Z^TJ_{2n}^{}X\\
		%&\quad + \rho XX^TJ_{2n} Z\skew(J_{2k}^TX^{T}\nabla\bar{f}(X)) +  \rho X\skew( J_{2k}^TX^{T}\nabla\bar{f}(X)) Z^TJ_{2n}^{}X\\
		& = \rho X\Bigl(\sym\big(X^TJ_{2n} Z\nabla\bar{f}(X)^TXJ_{2k} \big) 
		+\sym\big(X^TJ_{2n} Z J_{2k}^{T}X^{T}\nabla\bar{f}(X)\big)\Bigr).  \label{eq:second_term}
  \end{align}
 Taking $X_\perp$ satisfying \eqref{eq:choice_Xperp}, we obtain that
 \begin{equation}\label{eq:PXperp}
 J_{2n}^{}X_{\perp}^{}X_{\perp}^{T}J_{2n}^{T} = 
 \big(I_{2n}-XJ_{2k}^{}X^TJ_{2n}^T\big)\big(I_{2n}-XJ_{2k}^{}X^TJ_{2n}^T\big)^T=P_X^{} P_X^T
 \end{equation}
 with $P_X$ given in~\eqref{eq:prPXperp}, see \cite[Sect.~4.2]{GSAS21}. This implies
 \begin{equation}\label{eq:PiXPXperp}
     \Pi_X^\perp P_X^{} P_X^T = \Pi_X^\perp P_X^T = P_X^T
 \end{equation}
 and, hence,
 \begin{align}\label{eq:PXpprojMb}
	\PXp\projc\big(\Mb_{X,c,\rho}^{-1}\nabla\bar{f}(X)\big)
    =\PXp P_X^{} P_X^T\nabla\bar{f}(X)
	= P_X^T\nabla\bar{f}(X).
	\end{align}
 Then the third term in \eqref{eq:MinvK} takes the form
 \begin{equation}\label{eq:third_term}
		-\rho\, X\sym\big(Z^T\PXp\big(\projc\Mb_{X,c,\rho}^{-1}\nabla\bar{f}(X)\big)\big) = - \rho\, X\sym\big( Z^T P_X^T\nabla\bar{f}(X)\big).
	\end{equation}
Using \eqref{eq:PXperp} and $P_X^TJ_{2n}P_X^{}P_X^T=P_X^TJ_{2n}P_X^T$, the fourth term in \eqref{eq:MinvK} is expressed as
	\begin{align} 
 \begin{split}\label{eq:fourth_term}
		\frac{2}{\rho}\,J_{2n}&X_\perp^{} X_\perp^T\sym\big( Z\big(\projc(\Mb_{X,c,\rho}^{-1}\nabla\bar{f}(X))\big)^T\big)J_{2n}X\\ 
	 & = P_X^{} P_X^TJ_{2n}\big(Z\,\sym\big(\nabla\bar{f}(X)^TXJ_{2k}\big)
 + \frac{1}{\rho} P_X^T\nabla\bar{f}(X) Z^TJ_{2n}^{}X\big).
  \end{split}
	\end{align}
By direct calculation, we obtain that
	\begin{equation}\label{eq:K_fact5}
		J_{2n}X_\perp^{} X_\perp^TJ_{2n}^TX\big(X^TX\big)^{-1} = 
		 P_X^{} P_X^TX\big(X^TX\big)^{-1} = P_X J_{2n}^{}XJ_{2k}^{}.		
	\end{equation} 
	Combining it with \eqref{eq:PXpprojMb}, the fifth term in \eqref{eq:MinvK} is given by
\begin{align}\nonumber
			- J_{2n}&X_\perp^{} X_\perp^TJ_{2n}^TX\big(X^TX\big)^{-1}\sym\big( Z^T\PXp\projc\big(\Mb_{X,c,\rho}^{-1}\nabla\bar{f}(X)\big)\big) \\
	 & = - P_X J_{2n}^{}XJ_{2k}^{}\,\sym\bigl( Z^TP_X^T\nabla\bar{f}(X)\bigr).\qquad\qquad\qquad\qquad \qquad\qquad  \label{eq:fifth_term}
\end{align}
The sixth term can be determined from \eqref{eq:Rgrad_canon}, \eqref{eq:PXperp}, \eqref{eq:PiXPXperp} and \eqref{eq:K_fact5} as
\begin{align}\nonumber
-J_{2n}^{}&X_\perp^{} X_\perp^TJ_{2n}^T\PXp Z\,\skew\big((X^TX)^{-1}X^T\projc\big(\Mb_{X,c,\rho}^{-1}\nabla\bar{f}(X)\big)\big) \\
& = -P_X Z\, \skew\big(J_{2k}X^TJ_{2n}P_X^T\nabla\bar{f}(X) 
           + \rho\, J_{2k}\,\sym\big(J_{2k}^TX^T\nabla\bar{f}(X)\big)\big). \label{eq:sixth_term}
\end{align}
Finally, using \eqref{eq:Rgrad_canon} and \eqref{eq:PXpprojMb}, we compute the last term 
\begin{align}\nonumber
-J_{2n}^{}&X_\perp^{} X_\perp^TJ_{2n}^T\PXp \projc\big(\Mb_{X,c,\rho}^{-1}\nabla\bar{f}(X)\big)\,\skew\big((X^TX)^{-1}X^T Z\big)\qquad\qquad\qquad\\
& = - P_X^{} P_X^T\nabla\bar{f}(X)\,\skew\big((X^TX)^{-1}X^T Z\big). \label{eq:seventh_term}
\end{align}
Putting \eqref{eq:second_term}, \eqref{eq:third_term}, \eqref{eq:fourth_term}, \eqref{eq:fifth_term}, \eqref{eq:sixth_term}, and \eqref{eq:seventh_term} together, we get
\begin{align}\nonumber
		\Mb_{X,c,\rho}^{-1}&\mathcal{K}\big( Z,\projc(\Mb_{X,c,\rho}^{-1}\nabla\bar{f}(X))\big) 
        =  XJ_{2k}\,\skew\big( Z^TJ_{2n}^T\projc(\Mb_{X,c,\rho}^{-1}\nabla\bar{f}(X))\big)\\ \notag
		& + \rho X\Bigl(\sym\big(X^TJ_{2n} Z\nabla\bar{f}(X)^TXJ_{2k} \big) 
		-\sym\big( Z^T\nabla\bar{f}(X)\big)\Bigr)\\
		\notag
		& +P_X^{}P_X^TJ_{2n}^{}\Bigl( Z\,\sym\big(\nabla\bar{f}(X)^TXJ_{2k}\big)
            + \frac{1}{\rho}\,P_X^T\nabla\bar{f}(X) Z^TJ_{2n}^{}X\Bigr)\\ \notag
		& - P_X J_{2n}^{}XJ_{2k}^{}
			\,\sym\big( Z^TP_X^T \nabla\bar{f}(X)\big)\\[1mm] \notag
		& - P_X Z\,\skew\big(J_{2k}^{}X^TJ_{2n}^{}P_X^T 
			\nabla\bar{f}(X) + \rho\, J_{2k}\,\sym\big(J_{2k}^TX^T\nabla\bar{f}(X)\big)\big)\\[1mm] \label{eq:MKPM}
		&-P_X^{}P_X^T\nabla\bar{f}(X)\,\skew\big((X^TX)^{-1}X^T Z\big).
	\end{align}
	
	Next, let us turn to the two terms left in \eqref{eq:Rhessian_rewrite}. Combining \eqref{eq:Diff_B_X} and \eqref{eq:B_X_E_inverse},  we have
	\begin{align*}
	(\D_Z\Mb_{X,c,\rho})\Mb_{X,c,\rho}^{-1} 
      =&\  J_{2n}XZ^T\!J_{2n}^TXX^T\! 
      + J_{2n}^{}Z J_{2k}^TX^T\! 
      - \rho\,\PXp Z X^T\! \\
      & +\frac{1}{\rho}J_{2n}X Z^T\!X_\perp^{} X_{\perp}^TJ_{2n}^T 
	    - X(X^T\!X)^{-1} Z^T\PXp J_{2n}X_\perp^{}X_{\perp}^TJ_{2n}^T\! \\
     & - \PXp Z(X^T\!X)^{-1}X^T\! J_{2n}X_\perp^{} X_{\perp}^TJ_{2n}^T.
 	\end{align*}
	Due to \eqref{eq:PiXPXperp} and \eqref{eq:K_fact5},  this results  in
	\begin{align}\notag
		\Mb_{X,c,\rho}^{-1}&\left(\D_Z\Mb_{X,c,\rho}\right)\Mb_{X,c,\rho}^{-1}\nabla\bar{f}(X) =\Bigl(\rho\, X J_{2k}Z^TJ_{2n}^TXX^T + \rho\, XX^TJ_{2n}^{}Z J_{2k}^TX^T  \\ \notag
		&\quad 
        + XJ_{2k}Z^TJ_{2n}^{T}J_{2n}^{}X_\perp^{} X_{\perp}^TJ_{2n}^T 
        - \rho XZ^T \PXp J_{2n}^{}X_\perp^{}X_{\perp}^TJ_{2n}^T \\ \notag
		&\quad 
        + J_{2n}^{}X_\perp^{}X_{\perp}^TJ_{2n}^{T}J_{2n}^{}Z J_{2k}^TX^T 
        - \rho J_{2n}^{}X_\perp^{}X_{\perp}^TJ_{2n}^T\PXp Z X^T         \\ \notag
		&\quad - J_{2n}^{}X_\perp^{} X_{\perp}^TJ_{2n}^TX(X^TX)^{-1} Z^T \PXp J_{2n}^{}X_\perp^{}X_{\perp}^TJ_{2n}^T\\ \notag
		&\quad  - J_{2n}^{}X_\perp^{}X_{\perp}^TJ_{2n}^T\PXp Z (X^TX)^{-1}X^T J_{2n}^{}X_\perp^{}X_{\perp}^TJ_{2n}^T
		\Bigr)\nabla\bar{f}(X)\\\notag 
		&= 2\,\rho\,\sym\big(X J_{2k} Z^TJ_{2n}^TXX^T
         - XZ^TP_X^T\big)\nabla\bar{f}(X)\\ \notag
		&\quad + 2\,\sym\big(XJ_{2k} Z^TJ_{2n}^{T} P_X^{}P_X^T-
  P_X J_{2n}^{}XJ_{2k}^{}	Z^T  P_X^T \big)\nabla\bar{f}(X)\notag \\
		& = -2\,\sym\big(\rho\,XZ^T - (XJ_{2k}^{}Z^TJ_{2n}^T+ZJ_{2k}^{}X^TJ_{2n}^T)
  P_X^T \big)\nabla\bar{f}(X)\notag \\
        & = -2\,\sym\big(\rho\, XZ^T-2\,\skew(XJ_{2k}^{}Z^T)J_{2n}^T
        P_X^T\big)\nabla\bar{f}(X).\label{eq:MDMM}
	\end{align}
	Using \eqref{eq:B_X_E_inverse} and \eqref{eq:Diff_orth_proj_canonical}, we obtain
	\begin{align}\nonumber 
		\D_Z\projc(\Mb_{X,c,\rho}^{-1}\nabla \bar{f}(X)) = &-XJ_{2k}^{}\skew\big(Z^TJ_{2n}^T \Mb_{X,c,\rho}^{-1}\nabla \bar{f}(X)\big)\\
		&-\rho Z J_{2k}\skew\big(J_{2k}^TX^T\nabla \bar{f}(X)\big).
		\label{eq:DPMnab}
	\end{align}
	
	As the last step, inserting \eqref{eq:MKPM}, \eqref{eq:MDMM}, and \eqref{eq:DPMnab} into the general formulation for the Riemannian Hessian \eqref{eq:Rhessian_rewrite}, removing the terms belonging to the normal space  \eqref{eq:normal_space_canonical} and taking into account that 
 $\projc(P_X Y)= P_X Y$ for all $Y\in\R^{2n\times 2k}$, we complete the proof.
 
% \section*{Declaration}
% The authors have no competing interests to declare that are relevant to the content of this article.

%\begin{acknowledgements}
%If you'd like to thank anyone, place your comments here
%and remove the percent signs.
%\end{acknowledgements}

% BibTeX users please use one of
%\bibliographystyle{spbasic}      % basic style, author-year citations
\bibliographystyle{spmpsci}      % mathematics and physical sciences%\bibliographystyle{spphys}       % APS-like style for physics
\bibliography{references}   % name your BibTeX data base
% Non-BibTeX users please use
% \begin{thebibliography}{}
%
% and use \bibitem to create references. Consult the Instructions
% for authors for reference list style.
%
% \bibitem{RefJ}
% % Format for Journal Reference
% Author, Article title, Journal, Volume, page numbers (year)
% % Format for books
% \bibitem{RefB}
% Author, Book title, page numbers. Publisher, place (year)
% % etc
% \end{thebibliography}

\end{document}